\documentclass[reqno]{amsart}
\usepackage{amsmath,amsthm,amssymb,amsfonts,mathrsfs,amscd}
\usepackage{multirow}
\usepackage{graphicx}
\usepackage{bm}
\usepackage{mathrsfs}
\usepackage{dsfont}
\usepackage{tikz}
\usepackage{tikz-3dplot}
\usepackage{subfig}
\usepackage{stmaryrd}

\usepackage{float}
\usepackage[textfont=it]{caption}

\newtheorem{remark}{Remark}[section]

\newtheorem{lemma}{Lemma}[section]
\newtheorem{theorem}{Theorem}[section]

\newtheorem{definition}{Definition}[section]

\makeatletter
  \@addtoreset{equation}{section}
  \newcommand\figcaption{\def\@captype{figure}\caption}
  \newcommand\tabcaption{\def\@captype{table}\caption}
\makeatletter

\begin{document}

\title{A geometric optics ansatz-based plane wave method for two dimensional Helmholtz equations with variable wave numbers}

\author{Qiya Hu}
\author{Zezhong Wang}

\thanks{1. LSEC, ICMSEC, Academy of Mathematics and Systems Science, Chinese Academy of Sciences, Beijing
100190, China; 2. School of Mathematical Sciences, University of Chinese Academy of Sciences, Beijing 100049,
China (hqy@lsec.cc.ac.cn, bater1@yeah.net). This work was funded by Natural Science Foundation of China G12071469.}

\maketitle

{\bf Abstract.} In this paper we develop a plane wave type method for discretization of homogeneous Helmholtz equations with variable wave numbers. In the proposed method, local basis functions (on each element)
are constructed by the geometric optics ansatz such that they approximately satisfy a homogeneous Helmholtz equation without boundary condition. More precisely, each basis function is expressed as the product of an exponential plane wave function and a polynomial function, where the phase function in the exponential function
approximately satisfies the eikonal equation and the polynomial factor is recursively determined by transport equations associated with the considered Helmholtz equation. We prove that the resulting plane wave spaces have high order $h$-approximations as the standard plane wave spaces (which are available only to the case with constant wave number).
We apply the proposed plane wave spaces to the discretization of nonhomogeneous Helmholtz equations with variable wave numbers and establish the corresponding error estimates of their finite element solutions. We report
some numerical results to illustrate the efficiency of the proposed method.



\pagestyle{myheadings}
\thispagestyle{plain}
\markboth{}{}

\section{Introduction}
In this paper we consider the following Helmholtz equation with impedance boundary condition
\begin{equation}\label{nonhomogeneousHelm}
\left\{
	\begin{aligned}
	&\mathcal{L} u  = -(\Delta +\kappa^2(\mathbf{r})) u(\omega,\mathbf{r}) = f(\mathbf{r}), \quad \mathbf{r} = (x, y) \in\Omega, \\
	&(\partial_{\mathbf{n}} + i\kappa(\mathbf{r})) u(\omega,\mathbf{r}) = g(\mathbf{r}), \quad \mathbf{r}\in\partial\Omega,
	\end{aligned}
	\right.
\end{equation}
where $\Omega\subset \mathbb{R}^2$ is a bounded Lipchitz domain, $\mathbf{n}$ is the out normal vector on $\partial\Omega$, $f\in L^2(\Omega)$ is the source term and
$\kappa(\mathbf{r})={\omega\over c(\mathbf{r})}$,
$g\in L^2(\partial\Omega)$.
In applications, $\omega$ denotes the frequency and may be large, $c(\mathbf{r})>0$ denotes the light speed, which is usually a variable positive function.
The number $\kappa(\mathbf{r})$ is called the {\it wave number}.

Helmholtz equation is the basic model in sound propagation. It is a very important
topic to design a high accuracy method for Helmholtz equations with large wave numbers,
such that the so called {\it pollution effect} can be reduced. Let $V_h(\Omega)$ denote the finite element space of a finite element method for (\ref{nonhomogeneousHelm}), and let $\|\cdot\|_V$ denotes some ``energy" norm.
Assume that the number of basis functions on every element is fixed. The {\it pollution effect} says that the finite element solution $u_h$ does not satisfy the quasi-optimality
$$ \|u-u_h\|_V\leq C\inf_{v_h\in V_h(\Omega)}\|u-v_h\|_V $$
with a positive constant $C$, unless $h\omega^{1+\delta}$ is bounded for some positive number $\delta$ (in other words, $h\omega\rightarrow 0^+$ when $\omega\rightarrow+\infty$).
This means that the accuracies of the approximations may be destroyed when the wave number $\omega$ increases, unless the mesh sizes $h$ are chosen such that $h\omega=O(\omega^{-\delta})$, which converges zero as $\omega\rightarrow\infty$.
 For convenience, we call the parameter $\delta$ the {\it pollution index}, which describes the degree of pollution effect. For the standard linear finite element method, the {\it pollution index} $\delta=1$ (see \cite{fem}). There are some methods to reduce the {\it pollution index}, for example, the higher order finite element methods (refer to \cite{fem}).
To our knowledge, there seems no {\it pollution-free} (i.e., $\delta=0$) finite element method in literature except that a good approximation of the wave propagation directions are preliminarily known.

In order to compare different discretization methods, we would like to give a new concept. Let the number of basis functions on every element be fixed. A finite element method is called {\it weakly pollution-free} if
$$ \|u-u_h\|_{L^2(\Omega)}\leq C(f,g)(\omega h)^m, $$
where $C(f,g)$ is a constant independent of $\omega$ and $h$, but it may depend on the known functions $f$, $g$ and the number of basis functions on every element; the positive $m$
depends on the regularity of the analytic solution $u$ and the number of basis functions on every element. For a {\it weakly pollution-free} finite element method, the $L^2$ error $\|u-u_h\|_{L^2(\Omega)}$
does not obviously increase when fixing the value of $h\omega$ but increasing $\omega$.
It can be seen, from Theorem 3.15 of \cite{Hiptmair2011}, that the plane wave methods are {\it weakly pollution-free} for {\it homogeneous} Helmholtz equations (and time-harmonic Maxwell equations) with constant (or piecewise constant) wave numbers, provided that the known functions are independent of $\omega$ (which implies that the $\omega$-weighted norm $\|u\|_{k+1,\omega,\Omega}\leq \omega^k$).   
The main reason is that plane wave basis functions are solutions of a homogeneous Helmholtz equation without boundary condition and can capture at the maximum the oscillating characteristic of the analytic solution of the original Helmholtz equation. However, the other finite methods seem not {\it weakly pollution-free}. 

There are many articles to study the plane wave methods for these kinds of equations (see, for example, \cite{Buffa2008,Cessenat1998}, \cite{Farhat2003},\cite{Gabard2007}-\cite{Huqy2018}, \cite{Huttunen2007}, \cite{Monk1999},\cite{Perugia2016}). The plane wave methods have three drawbacks:
(i) the condition numbers of the resulting discrete systems increase too fast when the scales of the systems become large;
(ii) the standard plane wave methods are not {\it weakly pollution-free} for the nonhomogeneous Helmholtz equations;
(iii) the standard plane wave methods are not applicable to the Helmholtz equations with general variable wave numbers. The first drawback can be moved to some extent by constructing
efficient preconditioners for the discrete systems (see, for example, \cite{HuZ2016, Shu2018}). Fortunately, for the plane wave methods the condition numbers of the preconditioned systems can significantly decrease even if using simple domain decomposition preconditioners (see \cite{HuZ2016, HuZ2017}).
Recently (see \cite{Huqy2018,Yuan2019}), a plane wave method combined
with local spectral elements was proposed for the discretization of the nonhomogeneous Helmholtz equation (and
time-harmonic Maxwell equations) with (piecewise) constant wave numbers. The basic ideas in this method can be described as follows. At first nonhomogeneous Helmholtz equations on small subdomains
are discretized in the space consisting of higher order polynomials, then the resulting residue Helmholtz equation (which is homogeneous on each
element) on the global solution domain is discretized by the plane wave method. The method is also {\it weakly pollution-free} for the nonhomogeneous Helmholtz equations.
By using this method, we need only to study plane wave method for Helmholtz equations that are homogeneous on every element. Then
we can simply consider the homogeneous Helmholtz equation (on each element)
\begin{equation}\label{homohelm}
(\Delta +\kappa^2(\mathbf{r})) u(\omega,\mathbf{r}) =0, \quad \mathbf{r} = (x, y) \in\Omega.
\end{equation}

When $\kappa(\mathbf{r})$ is not a piecewise constant function (i.e., $c(\mathbf{r})$ is a general positive function), it is impossible in most cases to get ``exact" plane wave basis functions that are analytic solution of
(\ref{homohelm}) on one element. The first attempt to investigate plane wave method for the case of variable wave numbers was done in \cite{Imbert2014}, where basis functions were designed to locally satisfy an approximated version of
(\ref{homohelm}), and the resulting finite element method was called {\it general plane wave} (GPW) method.
The $h$-convergence of the approximate solutions generated by the GPW method was established in \cite{Imbert2015}. A modified Trefftz Discontinuous Galerkin (TDG) scheme of GPW methods for solving the homogeneous equation (\ref{homohelm}) was studied in \cite{Imbert2017}. The results given in \cite{Imbert2014,Imbert2015,Imbert2017} indicate that the GPW methods are not {\it weakly pollution-free} even for homogeneous Helmholtz equations.

In this paper
we design new kinds of plane wave type basis functions for the discretization of the Helmholtz equation (\ref{homohelm}) with a variable wave number $\kappa$. The key idea is to
use the {\it geometric optics ansatz}, i.e., the basis function is prior chosen to have the same formulation as the geometric optics ansatz of the wave field. Based on this idea, we construct basis functions such that each
of them is expressed as the product of an exponential plane wave function and a polynomial function, where the phase function in the exponential function
approximately satisfies the eikonal equation and the polynomial factor is recursively determined by transport equations derived by
(\ref{homohelm}). We prove that the plane wave spaces spanned by them possess high order $h$-approximate properties as the standard plane wave spaces (which are available only to the case of constant wave numbers). Then, by combining the ideas proposed in \cite{Huqy2018}, we apply the constructed plane wave spaces to the discretization of the nonhomogeneous Helmholtz equation (\ref{nonhomogeneousHelm}) and give the corresponding error estimates of the approximations, which indicate that the proposed methods are {\it weakly pollution-free}. We test several examples to confirm the efficiency of the proposed methods.

The paper is organized as follows. In Section \ref{GOPWs}, we derive the expressions of the new plane wave type basis functions. The approximate properties of the resulting finite element spaces
are proved in Section \ref{interperrest}. In Section \ref{variation}, we describe a PWDG type method combined with local spectral elements to solve nonhomogeneous Helmholtz equation (\ref{nonhomogeneousHelm}) and
give the corresponding error estimates. Finally, we report some numerical results of the proposed methods in Section \ref{numerical}.

\section{Geometric Optics Plane Wave (GOPW)}\label{GOPWs}

For a given $h>0$, we divide the domain $\Omega$ into a union of quasi-uniformly polygonal elements with the size $h$. Let $\mathcal{T}_h$ denote the resulting partition.
 The barycenter of an
element $K_0\in\mathcal{T}_h$ is denoted by $\mathbf{r}_0 = (x_0, y_0)\in\mathbb{R}^2$.
For convenience, we separate $\omega$ from $\kappa$: $\kappa(\mathbf{r})=\omega\sqrt{\xi(\mathbf{r})}$ with $\xi(\mathbf{r})= 1/c^2(\mathbf{r})$.

Let us recall the plane wave methods for Helmholtz equation (\ref{homohelm}). For a given positive integer $p$, a plane wave basis function $\varphi_l=e^{i\kappa\tau_l({\bf r})}$ ($1\leq l \leq p$) adapted to (\ref{homohelm}) satisfies $\mathcal{L} \varphi_l = 0$, and if $c(\mathbf{r})$ is a constant, we can choose $\tau_l({\bf r})={\bf d}_l\cdot{\bf r}$ with
${\bf d}_l=(\cos\theta_l,\sin\theta_l)$, which is called a plane wave direction. However, for the case that $c(\mathbf{r})$ is a variable function, such a function $\tau_l$ cannot be gotten directly. Because of this, it was proposed in \cite{Imbert2014} to construct $\varphi_l$ such that $\varphi_l$ has the form $\varphi_l=e^{P_l(\omega,{\bf r})}$ (where $P_l(\omega,{\bf r})$ is a polynomial, whose coefficients contain $\omega$) and
it locally satisfies an approximate version of the governing equation (\ref{homohelm})
\begin{equation} \label{homoapp}
	|\mathcal{L} \varphi_l|\leq  C h^{q}
\end{equation}
on every element, where $q$ is a given positive integer. The polynomial $P_l(\omega,{\bf r})$ determined by (\ref{homoapp}) cannot be written as $P_l(\omega,{\bf r})=i\kappa\tau_l({\bf r})$, so
the resulting {\it generalized plane wave} (GPW) method does not possess the desired approximation. If replacing $P_l$ by $i\kappa\tau_l({\bf r})$, then the positive number
$C$ in (\ref{homoapp}) depends on $\omega$, namely, $C=C(\omega)$, which is an increasing function of $\omega$, and the accuracies of the approximate solutions generated by the GPW method are destroyed by the ``bad" factor $C(\omega)$.

The purpose of this paper is to construct new plane wave basis functions by using the geometric optics ansatz such that the resulting discrete space possesses better approximation.

\subsection{Geometric optics ansatz}\label{goa}

Assume that the solution of the equation (\ref{homohelm}) can be approximated by a simple wave. According to the geometric optics ansatz \cite{Engquist2003}, the solution of (\ref{homohelm}) can be expressed as the WKJB approximation \cite{Jeffreys1925} (or the L{\"u}neberg-Kline expansion \cite{Kline1962}):
\begin{equation}\label{geoana}
	u(\omega,\mathbf{r}) = e^{i\omega\phi(\mathbf{r})} A(\mathbf{r}) ,
\end{equation}
where $\phi$ is called the phase function satisfying the eikonal equation
\begin{equation} \label{eik}
	|\nabla\phi(\mathbf{r})|^2 = \xi(\mathbf{r}),
\end{equation}
and $A(\mathbf{r})$ is called the amplitude function that can be written as
\begin{equation}
	A(\mathbf{r}) = \sum_{s = 0}^{\infty}\frac{A_s(\mathbf{r})}{(i\omega)^s}
\end{equation}
with $\{A_s\}_{s = 0}^{\infty}$ satisfying a recursive system of PDEs:
\begin{equation} \label{tran}
	2\nabla\phi\cdot\nabla A_s + A_s \Delta\phi = -\Delta A_{s-1}
\end{equation}
for $s = 0, 1, \cdots$, with $A_{-1} \equiv 0$.

The key features of the geometric optics ansatz are:
\begin{itemize}
	\item $\{A_s\}_{s = 0}^{\infty}$ and $\phi$ are independent of the frequency $\omega$;
	\item $\{A_s\}_{s = 0}^{\infty}$ and $\phi$ depend on $c(\mathbf{r})$ (and $f(\mathbf{r})$ if (\ref{nonhomogeneousHelm}) is considered).
\end{itemize}

When more waves are involved in the solution of the equation (\ref{homohelm}), the generic solution of (\ref{homohelm}) should be locally defined as a finite sum of terms like (\ref{geoana}). Hence, in general crossing waves, we use $N(\mathbf{r})$ to denote the number of crossing waves at the position $\mathbf{r}$ and expresse the solution of the Helmholtz equation (\ref{homohelm}) as
\begin{equation}\label{geoanaN}
	u(\omega,\mathbf{r}) = \sum_{n = 1}^{N(\mathbf{r})} u_n(\omega,\mathbf{r}),
\end{equation}
where each $u_n(\omega,\mathbf{r})$ has its ansatz form as (\ref{geoana})
\begin{equation}\label{geoanan}
	u_n(\omega,\mathbf{r}) =A_n({\bf r})e^{i\omega\phi_n(\mathbf{r})} = e^{i\omega\phi_n(\mathbf{r})}\sum_{s = 0}^{\infty}\frac{A_{n,s}(\mathbf{r})}{(i\omega)^s}.
\end{equation}
And for the $n$-wave ansatz ($n = 1,\cdots,N(\mathbf{r})$), the $\omega$-independent phase function $\phi_n(\mathbf{r})$ and $\{A_{n,s}(\mathbf{r})\}_{s = 0}^{\infty}$ satisfy the eikonal equation (\ref{eik}) and the corresponding system (\ref{tran}) respectively.

\subsection{Construction of plane wave type basis functions}
Motivated by the ansatz (\ref{geoana}), we define a plane wave type basis function as $\varphi(\mathbf{r})= a(\mathbf{r})e^{i\omega\tau(\mathbf{r})}$, where $\tau$ is a real polynomial approximately satisfying
(\ref{eik}) and $a$ is a complex polynomial approximately satisfying a transport equation derived by the governing equation (\ref{homohelm}) (setting $u=\varphi$).

Replacing $u$ in (\ref{homohelm}) with $\varphi = a(\mathbf{r})e^{i\omega\tau(\mathbf{r})}$, we get
$$ \mathcal{L} \varphi = [-\Delta a - i\omega(2\nabla a \cdot \nabla\tau + a\Delta\tau) + \omega^2a(\xi - |\nabla\tau|^2)]e^{i\omega\tau}=0, $$
which is equivalent to
$$ -\Delta a - i\omega(2\nabla a \cdot \nabla\tau + a\Delta\tau) + \omega^2a(\xi - |\nabla\tau|^2)=0. $$
When $\xi$ is a variable function, the functions $a$ and $\tau$ satisfying the above equation cannot be obtained directly. We consider a generic element $K_0$ with the diameter $h$ satisfying $\omega h=O(1)$,
and we try to find two polynomial functions $a$ and $\tau$ such that
$$ |-\Delta a - i\omega(2\nabla a \cdot \nabla\tau + a\Delta\tau) + \omega^2a(\xi - |\nabla\tau|^2)|\leq Ch^q,\quad \mbox{on}~~ K_0$$
for a given positive integer $q$. It is clear that the above inequality can be satisfied if
$$ |-\Delta a - i\omega(2\nabla a \cdot \nabla\tau + a\Delta\tau)|\leq Ch^q\quad\mbox{and}\quad |\xi - |\nabla\tau|^2|\leq Ch^{q+2},\quad \mbox{on}~~ K_0.$$
For convenience, we write the above inequalities in the form
\begin{subequations} \label{recdwhole}
	\begin{equation} \label{recd}
	\Delta a + i\omega(2\nabla a \cdot \nabla\tau + a\Delta\tau) = \mathcal{O}(h^{q})\quad \mbox{on}~~ K_0
	\end{equation}
and
	\begin{equation}\label{recdeik}
	\xi - |\nabla\tau|^2 = \mathcal{O}(h^{q+2})\quad \mbox{on}~~ K_0.
	\end{equation}
\end{subequations}

Let the polynomial $a$ and $\tau$ on $K_0$ be written as
\begin{equation}
	\begin{aligned}
	a(x,y) &= \sum_{i+j=0}^{m_a} a_{i,j} (x - x_0)^i (y - y_0)^j, \\
	\tau(x,y) &= \sum_{i+j=1}^{m_{\tau}} \lambda_{i,j} (x - x_0)^i (y - y_0)^j,
	\end{aligned}
\end{equation}
where $(x_0,y_0)$ is the barycenter of $K_0$. At first we can use the Taylor formula of $\xi$ to determine the coefficients of $\tau$ such that (\ref{recdeik}) is satisfied. Then we
use the expression of $\tau$ to determine the coefficients of $a$ such that (\ref{recd}) is met.
In the following two parts we give more details of the definitions of $a$ and $\tau$.

\subsubsection{Construction of $\tau(x,y)$.}\label{directau}

Consider an element $K_0$, and set ${\bf r}=(x,y)$ and ${\bf r}_0=(x_0,y_0)$. Assume that the function $\xi$ is smooth enough. We first let $q=1$ or $q=2$.
By the Taylor formula, the function $\xi$ can be expressed as
$$ \xi({\bf r})=T_{q+1}({\bf r}-{\bf r}_0)+\varepsilon_{q+2}({\bf r}-{\bf r}_0), \quad {\bf r}\in K_0,$$
where $T_{q+1}({\bf r}-{\bf r}_0)$ is the ($q+1$)-order Taylor polynomial of $x-x_0$ and $y-y_0$, the Taylor remainder $\varepsilon_{q+2}({\bf r}-{\bf r}_0)$ satisfies
$\varepsilon_{q+2}({\bf r}-{\bf r}_0)=\mathcal{O}(h^{q+2})$ (${\bf r}\in K_0$). We choose $m_{\tau}=q+2$ in the polynomial $\tau(x,y)$,
then $|\nabla \tau|^2$ is a polynomial of degree $2(q+1)$ of $x-x_0$ and $y-y_0$. Let $|\nabla \tau|^2$ be decomposed into
$$ |\nabla \tau|^2({\bf r})= \tilde{\tau}_{q+1}({\bf r}-{\bf r}_0)+\epsilon({\bf r}-{\bf r}_0), \quad {\bf r}\in K_0,$$
where $\tilde{\tau}_{q+1}$ is a ($q+1$)-order polynomial of $x-x_0$ and $y-y_0$, but $\epsilon$ satisfies $\epsilon({\bf r}-{\bf r}_0)=\mathcal{O}(h^{q+2})$ ~~$({\bf r}\in K_0)$.
It is clear that the requirement (\ref{recdeik}) is satisfied if
\begin{equation}
\tilde{\tau}_{q+1}({\bf r}-{\bf r}_0)=T_{q+1}({\bf r}-{\bf r}_0).\label{tau}
\end{equation}
Then we can compute the coefficients $\lambda_{i,j}$ of $\tau(x,y)$ by the above equation.

In fact, by the method of undetermined coefficients, the equation (\ref{tau}) can be written into the following equivalent algebraic systems
\begin{itemize}
	\item $i+j=1$: $\lambda_{1,0}^2 + \lambda_{0,1}^2 = \xi(\mathbf{r}_0)$,
	\item $i+j=2$:
\begin{equation*}
			\begin{bmatrix} 4\lambda_{1,0} & 2\lambda_{0,1} & 0 \\ 0 & 2\lambda_{1,0} & 4\lambda_{0,1} \\ \end{bmatrix}
			\begin{bmatrix} \lambda_{2,0}\\ \lambda_{1,1}\\ \lambda_{0,2}\\ \end{bmatrix}
			=
			\begin{bmatrix} \xi_{x}(\mathbf{r}_0)\\ \xi_{y}(\mathbf{r}_0)\\	 \end{bmatrix},
		\end{equation*}
	\item $i+j=3$:
\begin{equation*}
			\begin{bmatrix} 6\lambda_{1,0} & 2\lambda_{0,1} & 0 & 0 \\ 0 & 4\lambda_{1,0} & 4\lambda_{0,1} & 0 \\ 0 & 0 & 2\lambda_{1,0} & 6\lambda_{0,1} \\ \end{bmatrix}
			\begin{bmatrix} \lambda_{3,0}\\ \lambda_{2,1}\\ \lambda_{1,2}\\ \lambda_{0,3}\\ \end{bmatrix}
			=
			\begin{bmatrix} \frac{\xi_{xx}(\mathbf{r}_0)}{2} - 4\lambda_{2,0}^2 - \lambda_{1,1}^2\\ \xi_{xy}(\mathbf{r}_0)-4\lambda_{20}\lambda_{11}-4\lambda_{02}\lambda_{11}\\ \frac{\xi_{yy}(\mathbf{r}_0)}{2} - \lambda_{1,1}^2 - 4\lambda_{0,2}^2\\ \end{bmatrix}.
		\end{equation*}
\end{itemize}
For every integer $k>3$, the system determining all $\lambda_{i,j}$ for $i+j=k$ corresponds to a $k\times (k+1)$ coefficient matrix and can be similarly given (we omit the concrete form of these algebraic systems).

Notice that each coefficient matrix in the above algebraic systems is not square matrix, whose column is more than row, but is full-row-rank. Hence these systems have solutions but
the solutions $\{\lambda_{i,j}, i+j =k\}$ are not unique for $k=1,2,\cdots$.
Inspired by the construction of the classical plane wave basis functions, we choose $p$ plane wave directions $\{\mathbf{d}_l = (\cos\theta_l,\sin\theta_l)\}_{l = 1}^p$, and
define $(\lambda^l_{1,0},\lambda^l_{0,1}) = \sqrt{\xi(\mathbf{r}_0)}(\cos\theta_l,\sin\theta_l)$ ($l=1,\cdots, p$).
For each $(\lambda^l_{1,0},\lambda^l_{0,1})$, we can recursively compute a particular solution of the coefficients $\{\lambda_{i,j}, i+j =k\}$ for $k=2,3,\cdots, m_{\tau}$ by the above systems. Thus
we obtain $p$ different choices of the polynomial $\tau$ and we denote them by $\{\tau_l\}_{l = 1}^{p}$, which constitutes a set of $p$ independent phases.

As we will see that $q$ may weakly depend on $h$ and $\omega$, and $q$ may be large when $\omega h$ is very small. In this situation, we have to increase the order of the polynomial $\tau$ to
eliminate the effects of the factor $q$ in the coefficients of $\nabla \tau$.
Assume that $q$ satisfies $q\leq (\omega h)^{-{1\over 2}}$ when $\omega h\rightarrow 0$ (for a large $\omega$).
We write $\xi({\bf r})$ as
$$ \xi({\bf r})=T_{q+2}({\bf r}-{\bf r}_0)+\varepsilon_{q+3}({\bf r}-{\bf r}_0), \quad {\bf r}\in K_0,$$
where $T_{q+2}({\bf r}-{\bf r}_0)$ is the ($q+2$)-order Taylor polynomial of $x-x_0$ and $y-y_0$, the Taylor remainder $\varepsilon_{q+3}({\bf r}-{\bf r}_0)$ satisfies
$\varepsilon_{q+3}({\bf r}-{\bf r}_0)=\mathcal{O}(h^{q+3})$ (${\bf r}\in K_0$). We choose $m_{\tau}=q+3$ in the polynomial $\tau(x,y)$,
then $|\nabla \tau|^2$ is a polynomial of degree $2(q+2)$ of $x-x_0$ and $y-y_0$. Let $|\nabla \tau|^2$ be decomposed into
$$ |\nabla \tau|^2({\bf r})= \tilde{\tau}_{q+2}({\bf r}-{\bf r}_0)+\epsilon({\bf r}-{\bf r}_0), \quad {\bf r}\in K_0,$$
where $\tilde{\tau}_{q+2}$ is a ($q+2$)-order polynomial of $x-x_0$ and $y-y_0$, but $\epsilon$ satisfies (notice that the coefficients of $\nabla\tau$ contain $q$)
$$ \epsilon({\bf r}-{\bf r}_0)=\mathcal{O}(q^2h^{q+3})=\mathcal{O}(h^{q+2})\quad ({\bf r}\in K_0) $$
since $q\leq h^{-{1\over 2}}$. It is clear that the requirement (\ref{recdeik}) is satisfied if $\tilde{\tau}_{q+2}({\bf r}-{\bf r}_0)=T_{q+2}({\bf r}-{\bf r}_0)$.
The polynomial $\tau$ can be determined by the previous method.

\subsubsection{Construction of $a(x,y)$}\label{direca}

Since the left side of (\ref{recd}) contains $\omega$, the positive number $C$ in the bound of $\mathcal{O}(h^{q})$ in the right side of (\ref{recd}) generally depend on $\omega$.
Notice that the condition $h \omega=O(1)$ is a basic assumption in the numerical analysis of Helmholtz equations. Then we hope to construct a polynomial $a$ such that
$$  |\Delta a+i\omega(2\nabla a \cdot \nabla\tau + a\Delta\tau)|\leq C(h\omega)h^q, \quad \mbox{on}~K_0, $$
where $C(h\omega)$ is a positive number only depending on $h\omega$. To this end, we have to avoid to globally consider $\Delta a+i\omega(2\nabla a \cdot \nabla\tau + a\Delta\tau)$ for the
construction of the polynomial $a$.

At first we consider the cases with $q=1,2$. For these cases, we need only to construct a polynomial $a$ such that
	\begin{equation} \label{way1}
		|\Delta a|\leq Ch^{q}\quad\mbox{and} \quad |2\nabla a \cdot \nabla\tau + a\Delta\tau|\leq Ch^{q+1}.
	\end{equation}
We choose the order $m_a$ of polynomial $a$ as $m_a=q+1$ and transform the above two inequalities into two equations as in Subsection 2.2.1. Since $\Delta a$ is a polynomial of the $q-1$ degree, the first inequality
is equivalent to $\Delta a=0$. Moreover, the second inequality can be guaranteed if lower order terms contained in $2\nabla a \cdot \nabla\tau + a\Delta\tau$
vanish (the orders of the vanishing terms are less than $q+1$). Then we can use the method of undetermined coefficients
to determine two polynomials $a$ (when $q=1$) or one polynomial $a$ (if $q=2$) satisfying $a({\bf r}_0)=1$ by the derived two equations.

Next we consider the general case with $q\geq 3$. We find that the previous method is not applicable yet for this general case since the number of unknowns is less than the number of algebraic equations. Thus, for $q\geq 3$,
we have to use the recursive PDEs (\ref{tran}) to construct the desired polynomials $a$.

Choose $n_q=q-2$ or $n_q=q-1$ and define $a = \sum_{s= 0}^{n_q}\frac{a_s}{(i\omega)^s}$ with $a_s$ being a polynomial with the degree $q+1-s$. Then
\begin{eqnarray*}
\Delta a + i\omega(2\nabla a \cdot \nabla\tau + a\Delta\tau)&=&\sum_{s= 0}^{n_q}\frac{1}{(i\omega)^s}\Delta a_s+\sum_{s = 0}^{n_q}\frac{iw}{(i\omega)^s}(2\nabla a_s \cdot \nabla\tau + a_s\Delta\tau)\cr
&=&\sum_{s = 0}^{n_q-1}\frac{1}{(i\omega)^s}\Delta a_s+\sum_{s = 1}^{n_q}\frac{iw}{(i\omega)^s}(2\nabla a_s \cdot \nabla\tau + a_s\Delta\tau)\cr
&+&\frac{1}{(i\omega)^{n_q}}\Delta a_{n_q}+i\omega(2\nabla a_0 \cdot \nabla\tau + a_0\Delta\tau)\cr
&=&\sum_{s= 1}^{n_q}\frac{1}{(i\omega)^{s-1}}(2\nabla a_s \cdot \nabla\tau + a_s\Delta\tau+\Delta a_{s-1})\cr
&+&i\omega(2\nabla a_0 \cdot \nabla\tau + a_0\Delta\tau)+\frac{1}{(i\omega)^{n_q}}\Delta a_{n_q}.
\end{eqnarray*}
Then the condition (\ref{recd}) can be satisfied if $a_s$ ($s=0,1,\cdots,n_q$) are recursively determined by
$$
			2\nabla a_0 \cdot \nabla\tau + a_0 \Delta\tau = \mathcal{O}(h^q\omega^{-1})
$$
and (if~~$n_q\geq 1$)
$$
			2\nabla a_s \cdot \nabla\tau + a_s \Delta\tau + \Delta a_{s-1} = \mathcal{O}(h^q\omega^{s-1}), \quad s=1,2,\cdots, n_q
$$
and
$$			\Delta a_{n_q} = \mathcal{O}(h^q\omega^{n_q}).
$$
Notice that $q$ may be large when $\omega h$ is very small, and the coefficients in the left rights of the above relations contain the factors as $q^2$, so we need to eliminate the effects of the factors $q^2$.
For $s=1,2,\cdots, n_q$, let $q_s$ denote the smallest positive integer satisfying $q^2 h^{q_s}\leq h^q\omega^{s-1}$. Moreover, let $q^*$ denote the smallest positive integer satisfying $q^2 h^{q^*}\leq h^q\omega^{n_q}$.
Notice that $\omega h=O(1)$, we have $q_s\geq q+1-s$ and $q^*\geq q-n_q$. The previous three conditions can be deduced by (since $q^2h\leq \omega^{-1}$)
	\begin{subequations}

		\begin{equation} \label{a0}
			2\nabla a_0 \cdot \nabla\tau + a_0 \Delta\tau = \mathcal{O}(q^2h^{q+1})
		\end{equation}
and (if~~$n_q\geq 1$)
		\begin{equation}\label{as}
			2\nabla a_s \cdot \nabla\tau + a_s \Delta\tau + \Delta a_{s-1} =\mathcal{O}(q^2h^{q_s})
\quad s=1,2,\cdots, n_q
		\end{equation}		
and
		\begin{equation}\label{aq}
			\Delta a_{n_q} =\mathcal{O}(q^2h^{q^*})
		\end{equation}		
	\end{subequations}
Noting that the order of the polynomial $\Delta a_{n_q}$ is $q-n_q-1$, the equality (\ref{aq}) is equivalent to the equation
$$ 	\Delta a_{n_q} =0. $$

The equalities (\ref{a0}) and (\ref{as}) need to be transformed into two systems of equations as in Subsection 2.2.1. For example, the equations of (\ref{as}) are defined such that
all the polynomial terms whose orders are less than $q_s$ vanish.

If we only use these equations
to recursively compute $a_r$, the freedom degrees of $a_r$ may be very large. For example, the number of coefficients of $a_0$ is $(q+2)(q+3)/2$, but there are only $(q+1)(q+2)/2$ equations in (\ref{a0}). Because of this, we add extra constrains to reduce the degrees of freedom of $a_s$ ($s=0,\cdots, n_q-1)$:
\begin{equation}
	\left\{\begin{aligned}
		& (\Delta a_s)|_{{\bf r}=\mathbf{r}_0}= 0,\\
		& (\partial_{x}^{k_1} \partial_{y}^{k_2} \Delta a_s)(\mathbf{r}_0)= 0, \quad \forall k_1 + k_2 =q_s-2\quad(\mbox{if}~q_s>2) 
.
	\end{aligned}\right.
	\end{equation}
Hereafter, we simply use $\partial_{x}^{l} \partial_{y}^{j} \phi(\mathbf{r}_0)$ to denote ${\partial^{r+j} \phi\over \partial x^r\partial y^j}\mid_{(x_0,y_0)}$
for a smooth function $\phi$.
Then we can recursively compute polynomials $a_s$ (with the degree $(q+1-s)$) by  (\ref{a0})-(\ref{aq}) (we first transform them into equations as in Subsection 2.2.1),
together with the above constrains, and so we construct a polynomial $a$ by $a = \sum_{l = 0}^{n_q}\frac{a_s}{(i\omega)^s}$. 	

For $s=0,1,\cdots,q-2$, the obtained polynomial $a_s$ is unique. However, when $n_q=q-1$, the obtained polynomial $a_{q-1}$ has two different choices (i.e., $a_{q-1}$ has two freedom degrees).
This means that the polynomial $a$ has one degree of freedom (rep. two degrees of freedom) when $n_q=q-2$ (rep. $n_q=q-1$).

\subsection{Discrete spaces and their approximate properties}
Based on the discussions in the previous two subsections, we can give definitions of the geometric optics ansatz plane wave (GOPW) basis functions.

Considering a generic element $K_0\in\mathcal{T}_h$. Let $p$ and $q$ be two given positive integers. Suppose that $\xi\in\mathcal{C}^{q+1}(K_0)$. For direction angles $\theta_l\in[0,2\pi]$ ($l=1,\cdots,p$),
we construct phase polynomials $\tau_l(x,y)$ ($l=1,\cdots,p$) with the degree $q+2$ as in Subsection \ref{directau}. It is clear that the phase polynomials $\tau_l(x,y)$ are independent of $\omega$. Moreover,
any order derivatives of the amplitude polynomials $a_l(x,y)$ determined in Subsection \ref{direca} are uniformly bounded with respect to $\omega$.

\begin{definition}\label{GOPWbasis} The GOPW basis functions on the element $K_0$ are defined as follows:
\begin{itemize}
\item {\bf Case 1.} For $q\geq 2$, we construct an amplitude polynomial $a^q_l$ associated with the phase polynomial $\tau^q_l(x,y)$ as in Subsection \ref{direca} with the first terminate condition $n_q=q-2$.
There are $p$ GOPW basis functions on $K_0$:
\begin{equation}\label{bas2}
		\varphi^q_l(x,y)= a^q_l(x,y)e^{i\omega\tau^q_l(x,y)}, \quad l=1,\cdots,p \quad ((x,y)\in K_0).
	\end{equation}
	\item {\bf Case 2.} For $q\geq 1$, there are $2p$ GOPW basis functions on $K_0$:
	\begin{equation} \label{bas1}
		\psi^q_{l,j}(x,y) = a_{l,j}^q(x,y) e^{i\omega\tau^q_l(x,y)}, \quad l=1,\cdots, p; ~j = 1, 2\quad ((x,y)\in K_0),
	\end{equation}
	where the amplitude polynomials $a_{l,j}^q$ ($j=1,2$) correspond to the phase polynomial $\tau_l(x,y)$ and are recursively constructed as in Subsection \ref{direca} with the second terminate condition $n_q=q-1$.	
\end{itemize}
\end{definition}

\begin{remark}
Considering the case $\xi$ is constant in each element of $\mathcal{T}_h$, we can still obtain the discrete phase $\tau(x,y)$ and amplitude $a(x,y)$ by the procedures in the last subsection for $q=1$. In this case, the equation $\mathcal{L}\varphi = 0$ can be exactly satisfied in every element. Taking the second kind of construction procedure illustrated above, two basis function $\{\psi_{l,j}, j = 1, 2\}$ for each direction $\mathbf{d}_l$ can be obtained and they have the following form:
$$
	\psi_{l,1} = e^{i\omega \mathbf{d}_l\cdot\mathbf{r}} \quad\mbox{and}\quad \psi_{l,2} = \mathbf{d}_l^{\perp}\cdot\mathbf{r}e^{i\omega \mathbf{d}_l\cdot\mathbf{r}},
$$
which are very different from the standard plane wave basis functions.
\end{remark}

With the GOPW basis functions described in Definition 2.1, we define two GOPW finite element spaces on $K_0$ as
$$ V^{(1)}_{p,q}(K_0)=span\{\varphi^q_l:~l=1,\cdots,p\};~~V^{(2)}_{p,q}(K_0)=span\{\psi^q_{l,j}:~l=1,\cdots,p;~j=1,2\}. $$
Then every functions in the above spaces approximately satisfy the homogeneous Helmholtz equation on the element $K_0$ in the sense that
\begin{equation}
 | \Delta v +\kappa^2(\mathbf{r})v|\leq C|v|h^{q}\quad\mbox{on}~~K_0\quad\forall v\in V^{(r)}_{p,q}(K_0)~~~(r=1,2), \label{2.defin}
 \end{equation}
where $C$ is a constant independent of $\omega$ and $h$ provided that $q\leq 2$ or $q\leq (\omega h)^{-{1\over 2}}$ for small $\omega h$, $|v|$ denotes the maximal modulus of the coordinates of $v$ under the GOPW basis functions.

The following two theorems give approximate properties of the GOPW finite element spaces $V^{(1)}_{p,q}(K_0)$ and $V^{(2)}_{p,q}(K_0)$, respectively.

\begin{theorem} \label{interp} For a given integer $n\geq 2$, set $p=2n+1$. Let the mesh size $h$ satisfy $h\omega\leq C_0$ and choose $q=\max\{2, [{(n-4)\ln (\omega h)^{-1}\over \ln\omega}]\}$.
For an element $K_0$, suppose that $\xi \in C^{q+1}(K_0)$ and $u\in C^{n+1}(K_0)$, which
satisfies the equation $\mathcal{L} u = 0$ on $K_0$ and has the stability
\begin{equation}\label{2.newregul}
 \|u\|_{C^k(K_0)}\leq C(\xi)\omega^k,\quad  0\leq k\leq n+1.
\end{equation}
Then there exists a function $u_p\in V^{(1)}_{p,q}(K_0)$
such that
\begin{equation}\label{intererr}
\|u-u_p\|_{\infty,K_0} \leq C(\xi, n) h^{n+1}\omega^{n+1}.
\end{equation}
When $u\in C^{k}(K_0)$ for $3\leq k\leq n+1$, we have more general estimates
\begin{equation}\label{2.new-intererr}
\|u-u_p\|_{C^l(K_0)} \leq C(\xi, n)h^{k-l}\omega^k,~~~l=0,1.
\end{equation}
Here $C(\xi, n)$ is a positive number independent of $h$ and $\omega$.
\end{theorem}

\begin{theorem} \label{interp2} For a given integer $n\geq 2$, set $p=2n+1$. Let the mesh size $h$ satisfy $h\omega\leq C_0$ and
choose $q=\max\{1, [{(2n-4)\ln (\omega h)^{-1}\over \ln\omega}]\}$. For an element $K_0$, suppose that $\xi \in C^{q+1}(K_0)$ and $u\in C^{2n+1}(K_0)$, which
satisfies the equation $\mathcal{L} u = 0$ on $K_0$ and has the stability
\begin{equation}\label{2.newregul1}
 \|u\|_{C^k(K_0)}\leq C(\xi)\omega^k,\quad  0\leq k\leq 2n.
\end{equation}
Then there is a function $u_p\in V^{(2)}_{p,q}(K_0)$ such that
\begin{equation}\label{intererr1}
\|u-u_p\|_{\infty,K_0} \leq C(\xi, n) h^{2n+1}\omega^{2n+1}
\end{equation}
and (if $u\in C^{k}(K_0)$ for $3\leq k\leq 2n+1$)
\begin{equation}\label{2.new-intererr1}
\|u-u_p\|_{C^l(K_0)} \leq C(\xi, n)h^{k-l}\omega^k,~~~l=0,1.
\end{equation}
Here $C(\xi, n)$ is a positive number independent of $h$ and $\omega$.

The proofs of the above results are technical and will be stated in the next section.
\end{theorem}
\begin{remark} In the above two theorems, the factors $C(\xi, n)$ are independent of $\omega$, provided that $\omega h$ is upper bounded (which is a basic assumption in numerical
methods for Helmholtz equations). This means that the proposed GOPW methods have better approximate properties than the GPW method (compare Theorem 1 in \cite{Imbert2015}), which will be
confirmed by numerical experiments.
\end{remark}

\begin{remark} In Theorem 2.1 and Theorem 2.2 the parameter $q>2$ (rep. $q>1$) only when $n\geq 5$ and $h<(\omega^{-1})^{{n-1\over n-4}}$ (resp. $n\geq 3$ and $h<(\omega^{-1})^{{2n-2\over 2n-4}}$). The choices of $q$ means that every basis function should approximately satisfy the considered homogeneous Helmholtz equation
with slightly higher accuracy when $h\ll\omega^{-1}$ and $n\geq 5$ (resp. $n\geq 3$) such that the approximation $u_p$ possesses the desired $h$-convergence order. Of course, for the case of constant wave number such a condition can be ignored, since every basis function exactly satisfies the homogeneous Helmholtz equation. It is easy to verify that the chosen parameters $q$ satisfy $q\leq (\omega h)^{-{1\over 2}}$
when $\omega h\rightarrow 0$ and $n\geq 5$ (resp. $n\geq 3$).
We emphasize that the dimensions of the discrete spaces $V^{(1)}_{p,q}(K_0)$ and $V^{(2)}_{p,q}(K_0)$ do not depend on
the values of $q$.
\end{remark}

\section{The verification of the approximate properties}\label{interperrest}
In this section we are devoted to the derivations of Theorem \ref{interp} and Theorem \ref{interp2}.
For simplicity of exposition, we only give the details of the proof of Theorem \ref{interp} (Theorem \ref{interp2} can be proved in almost the same way). We follow the basic
ideas of the analysis for the generalized plane waves introduced in \cite{Imbert2015}, but we have to establish some new techniques in this section.

\subsection{Construction of the desired approximations $u_p$}

The main ideas are to establish suitable algebraic relations between the Taylor expansions of the GOPW
basis functions, classical plane wave basis functions as well as the analytic solution of equation (\ref{homohelm}).
\subsubsection{The basic ideas}
Let $p=2n+1$ with an integer $n\geq 2$, and choose $p$ direction angles $\theta_{l}=2\pi(l-1)/p$ ($l=1,\cdots,p$). For $q\geq 2$, suppose $\xi\in \mathcal{C}^{q+1}(K_0)$.
Consider the plane wave basis function ($\mathbf{r}_0=(x_0,y_0)$)
\begin{equation*}
	e_{l}(x, y)=e^{i \omega \xi(\mathbf{r}_0)((x-x_{0}) \cos \theta_{l}+(y-y_{0}) \sin \theta_{l})}
\end{equation*}
and the corresponding GOPW basis function (for ease of notation, we omit the upper index $q$)
\begin{equation*}
\varphi_{l}(x, y)=a_le^{i\omega\tau_l(x, y)}=(\sum_{s = 0}^{q-2}(i\omega)^{-s} a_{l,s})e^{i\omega\tau_l(x, y)},
\end{equation*}
where $\tau_l(x, y)$ are constructed as in Subsection 2.2.1, and $a_{l,s}$ are constructed as (\ref{a0})-(\ref{aq}) with replacing $\tau(x,y)$ by $\tau_l(x,y)$ ($n_q=q-2$).
For convenience, we write $\varphi_{l}(x, y)$ as
$$ \varphi_{l}(x, y)=\sum_{s = 0}^{q-2}(i\omega)^{-s} \varphi_{l,s}(x, y) , \quad \varphi_{l,s}(x, y) = a_{l,s}(x, y)e^{i\omega\tau_l(x, y)}. $$

We want to find a function $u_p\in V^{(1)}_{p,q}(K_0)$ such that $u_p$ can sufficiently approximate the analytic solution $u$ of the equation (\ref{homohelm}). Define
$$  u_p(x,y)=\sum\limits_{l=1}^p z_l\varphi_l(x,y). $$
It suffices to determine the coefficients $\{z_l\}_{l=1}^p$. To this end, we use $T_n(x,y)$ and $P_{l,n}(x,y)$ to denote the $n$-order Taylor polynomials of $u$ and $\varphi_l$
at the point $(x_0,y_0)\in K_0$, respectively, i.e.,
$$ T_n(x,y)=\sum_{k=0}^{n} \sum_{r+j=k}{\partial_x^r\partial_y^ju({\bf r}_0)\over r!j!}(x-x_0)^r(y-y_0)^j, $$
$$ P_{l,n}(x,y)=\sum_{k=0}^{n} \sum_{r+j=k}{\partial_x^r\partial_y^j\varphi_l({\bf r}_0)\over r!j!}(x-x_0)^r(y-y_0)^j. $$
Then, for $(x,y)\in K_0$ we have
\begin{equation}\label{utaylor}
	\left\{\begin{aligned}
	& |u(x, y)-T_n(x,y)| \leq {2^{n+1}\over (n+1)!}|\mathbf{r}-\mathbf{r}_0|^{n+1}\|u\|_{C^{n+1}(K_0)}, \\
	& |\varphi_{l}(x, y)-P_{l,n}(x,y)| \leq {2^{n+1}\over (n+1)!}|\mathbf{r}-\mathbf{r}_0|^{n+1}\|\varphi_{l}\|_{C^{n+1}(K_0)}.	\end{aligned}\right.
\end{equation}
A natural idea is to find $\{z_l\}_{l=1}^p$ such that
\begin{equation}
\sum\limits_{l=1}^p z_lP_{l,n}(x,y)=T_n(x,y),\quad (x,y)\in K_0. \label{3.1new}
\end{equation}
Unfortunately, we will realize that the above equation has no solution.
\subsubsection{The difficulties} At first we give an ordering rule of entries in a column vector composed of numbers with double indices.
Let $\mathbb{N}$ denotes the set of non-negative integers, set $m_j = (j+1)(j+2)/2$ for $j \in \mathbb{N}$. For a set $\{\chi_{r,j}:~0\leq r+j\leq n\}$
of numbers with double indices $(r,j)\in \mathbb{N}^2$, we use the following rule to define a column vector
$$ \alpha_n=(\chi_{r,j}:~0\leq r+j\leq n)=(\chi_{0,0}~~\chi_{1,0}~~\chi_{0,1}\cdots\chi_{n,0}~~\chi_{n-1,1}\cdots \chi_{0,n})^t. $$
It is clear that the vector $\alpha_n$ has the dimension $m_n$. In order to describe the above rule in a single index instead of double indices, we establish an one-to-one mapping from a
double indices $(r,j)$ ($r,j\in \mathbb{N}$) to a single index $k\in \mathbb{N}^+$. If $r=k=0$, then $F(r,j)=1$; otherwise, define
$$ k=F(r,j)=m_{r+j-1}+j+1,\quad r,j\in \mathbb{N};~r+j\geq 1. $$
It is easy to see that
$$ m_{r+j-1}+1\leq F(r,j)\leq m_{r+j-1}+(r+j)+1=m_{r+j}. $$
Then $1\leq F(r,j)\leq m_n$ when $0\leq r+j\leq n$. We can check that the inverse mapping
$F^{-1}$ can be described as $(0,0)=F^{-1}(1)$ and, for $2\leq k\leq m_n$,
$$ F^{-1}(k)=(r,j)=(m_l-k, k-(m_{l-1}+1)),\quad m_{l-1}+1\leq k\leq m_l; l=1,\cdots,n. $$
In the previous definition of the column vector $\alpha_n=(\chi_{r,j}:~0\leq r+j\leq n)$, the number $\chi_{r,j}$ is arranged as the $k$-row of $\alpha_n$ with $k=F(r,j)$.

Then we define matrices associated with coefficients of Taylor polynomials. For convenience, set $c^{(l)}_{r,j}={\partial_x^r\partial_y^j\varphi_l({\bf r}_0)\over r!j!}$.
Similarly, we use $c^{(e_l)}_{r,j}$ and $c^{(l,s)}_{r,j}$ to denote the coefficients in the Taylor formulas of $e_l(x,y)$ and $\varphi_{l,s}$ respectively
(i.e., $\varphi_l(x,y)$ in the above formula is replaced by $e_l(x,y)$ and $\varphi_{l,s}$).

\begin{definition}\label{matrix} Define $m_n\times p$ matrices ${\mathcal C}_{(n)}$ such that, for $l=1,\cdots,p$, the $l$-column of ${\mathcal C}_{(n)}$ consists of the numbers $c^{(l)}_{r,j}$ as the
rule described above, i.e., $c^{(l)}_{r,j}$ is arranged as the $k$-row of ${\mathcal C}_{(n)}$ with $k=F(r,j)$. Similarly, define $m_n\times p$ matrices ${\mathcal C}^{e}_{(n)}$ and ${\mathcal C}^s_{(n)}$
composed of the numbers $c^{(e_l)}_{r,j}$ and $c^{(l,s)}_{r,j}$, respectively.
\end{definition}

It is clear that
\begin{equation} \label{first-decomp}
	{\mathcal C}_{(n)} = \sum_{s = 0}^{q-2}(i\omega)^{-s} {\mathcal C}^s_{(n)}.
\end{equation}

Let $b^{(n)}$ denote the vector composed of the coefficients ${\partial_x^r\partial_y^ju({\bf r}_0)\over r!j!}$ of $T_n(x,y)$, which are arranged
as the same order described above. Then the equation (\ref{3.1new}) is equivalent to the following algebraic system
\begin{equation}\label{equivlinear}
{\mathcal C}_{(n)}Z^{(n)}=b^{(n)},
\end{equation}
where $Z^{(n)}=(z_1,\cdots, z_p)^t$ ($p=2n+1$). We will explain that this system has no solution.

Let $\mathcal{L}$ be the Helmholtz operator and $\varphi\in C^{n+1}(K_0)$. For ${\bf r}_0\in K_0$, a smooth function $\varphi$ is called {\it $\mathcal{L}$-vanishing} at ${\bf r}_0$ if
$$ (\partial^r_x\partial^j_y\mathcal{L}\varphi)({\bf r}_0)=0,\quad\mbox{for~all}~(r,j)~\mbox{satisfying}~0\leq r+j\leq n-2. $$
It is easy to see that such a function satisfy ${n(n-1)\over 2}$ constraints.

For convenience, the $m_n$ dimensional column vector $\alpha_n=(\chi_{r,j}:~0\leq r+j\leq n)$ consisting of the Taylor coefficient $\chi_{r,j}$ of $\varphi$ at ${\bf r}_0$ (i.e, $\chi_{r,j}=(\partial^r_x\partial^j_y\varphi)(x_0,y_0)/r!j!$) is called {\it $\varphi$-derived} vector at ${\bf r}_0$.

As in \cite{Imbert2015}, we define a particular vector space. Let $\mathbb{C}^{k}$ denote the space of complex column vector with the dimension $k$. For ${\bf r}_0\in K_0$, define a vector space $\mathbb{S}_n\subset \mathbb{C}^{m_{n}}$ by
\begin{equation}\label{Rspace}
	\begin{aligned}
	&\mathbb{S}_n:=\{\alpha=(\chi_{l,j}:~0\leq l+j\leq n)\in \mathbb{C}^{m_{n}};~\mbox{when}~k_1+k_2\leq n-2,~\mbox{we~have}\\
	&\quad 2m_{k_1}\chi_{k_{1}+2, k_{2}} + 2m_{k_2}\chi_{k_{1}, k_{2}+2} + \omega^2 \sum_{l=0}^{k_{1}} \sum_{j=0}^{k_{2}} \frac{(\partial_{x}^{l} \partial_{y}^{j} \xi)(\mathbf{r}_0)}{l ! j !}\chi_{k_{1}-l, k_{2}-j} = 0\} .
	\end{aligned}
\end{equation}
Noting that there are ${n(n-1)\over 2}$ constraints in the space $\mathbb{S}_n$, the dimensions of the space $\mathbb{S}_n$ are just $p=2n+1$ since
$$ \operatorname{dim} \mathbb{S}_n=m_n-{n(n-1)\over 2}={(n+1)(n+2)\over 2}-{n(n-1)\over 2}=2n+1. $$
\begin{lemma}\label{3.newlemma} A function $\varphi$ is {\it $\mathcal{L}$-vanishing} at ${\bf r}_0$ if and only if the {\it $\varphi$-derived} vector belongs to the space $\mathbb{S}_n$.
\end{lemma}
\begin{proof} Let double indices $(k_1,k_2)$ satisfy $0\leq k_1+k_2\leq n-2$. By the definition of the operator $\mathcal{L}$ and the high order derivative formula for the product of two functions, we have
\begin{eqnarray*}
\partial_x^{k_1}\partial_y^{k_2} \mathcal{L}\varphi& =&\partial_x^{k_1+2}\partial_y^{k_2} \varphi + \partial_x^{k_1}\partial_y^{k_2+2} \varphi \cr
	&+& k_1!k_2!\omega^2 \sum_{l=0}^{k_{1}} \sum_{j=0}^{k_{2}} \frac{\partial_{x}^{l} \partial_{y}^{j} \xi}{l ! j !} \frac{\partial_x^{k_1-l}\partial_y^{k_2-j} \varphi}{(k_1-l)!(k_2-j)!}.
\end{eqnarray*}
Then,
\begin{eqnarray}
{1\over k_1!k_2!}(\partial_x^{k_1}\partial_y^{k_2} \mathcal{L}\varphi)(\mathbf{r}_0)& =& (k_1+2)(k_1+1)\chi_{k_{1}+2, k_{2}}+ (k_2+2)(k_2+1)\chi_{k_{1}, k_{2}+2} \cr
	&+& \omega^2 \sum_{l=0}^{k_{1}} \sum_{j=0}^{k_{2}} \frac{(\partial_{x}^{l} \partial_{y}^{j} \xi)(\mathbf{r}_0)}{l ! j !} \chi_{k_{1}-l, k_{2}-j}.
\label{3.new-vanishing}
\end{eqnarray}
This equality implies the desired result.
\end{proof}

Since $\mathcal{L} u=0$, we have $b^{(n)} \in \mathbb{S}_n$ by Lemma \ref{3.newlemma} (see also \cite{Imbert2015}).

Now we investigate the solvability of (\ref{equivlinear}). By the construction of the GOPW basis function $\varphi_l$, we deduce that
\begin{eqnarray}
	\mathcal{L}\varphi_l & =-\{\sum\limits_{s = 0}^{q-2}(i\omega)^{1-s}(2\nabla a_{l,s} \cdot \nabla\tau_l + a_{l,s} \Delta\tau_l + \Delta a_{l,s-1})\}e^{i\omega\tau_l} - (i\omega)^2(|\nabla\tau_l|^2 - \xi)\varphi_l \cr
	& =-q^2\sum\limits_{s = 0}^{q-2}\{(i\omega)^{1-s} \varepsilon_{q+1-s}(x-x_0, y-y_0)\}e^{i\omega\tau_l} - (i\omega)^2\varepsilon_{q+2}(x-x_0, y-y_0)\varphi_l\cr
   &=-q^2\sum\limits_{r = 3}^{q+1}\{(i\omega)^{r-q} \varepsilon_r(x-x_0, y-y_0)\}e^{i\omega\tau_l} - (i\omega)^2\varepsilon_{q+2}(x-x_0, y-y_0)\varphi_l, \label{3.new-equal}
\end{eqnarray}
where $\varepsilon_j(x-x_0, y-y_0)$ is a polynomial of $x-x_0$ and $y-y_0$, each term of which contains a factor like $(x-x_0)^s(y-y_0)^t$ with $s+t=j$.
It is easy to see that
$$(\partial_x^j\partial_y^k\varepsilon_r(x-x_0,y-y_0))(r_0)=j!k!\not=0$$
if and only if $j+k=r$.
Thus, by (\ref{3.new-equal}) we deduce that (noting $\omega h=O(1)$)
\begin{equation} \label{diffH-new}
	\partial_x^{k_1}\partial_y^{k_2}\mathcal{L}\varphi_l(\mathbf{r}_0) =
	\begin{cases} {0, \quad  k_1+k_2 \leq 2} \\{\mathcal{O}(q^2\omega^{k_1+k_2-q})\neq 0,\quad 2<k_1+k_2 \leq n-2.} \end{cases}
\end{equation}

Throughout this paper, we use $\operatorname{Col}({\mathcal A})$ to denote the set of the column vectors of a $m_n\times p$ matrix $A$ and use $Im({\mathcal A})$
to denote the image space of ${\mathcal A}$, which is spanned by the vectors in $\operatorname{Col}({\mathcal A})$.

It follows by (\ref{diffH-new}) and Lemma \ref{3.newlemma} that $\operatorname{Col}({\mathcal C}_{(n)})\nsubseteq \mathbb{S}_n$, namely, $Im({\mathcal C}_{(n)})\nsubseteq \mathbb{S}_n$.
Notice that $b^{(n)} \in \mathbb{S}_n$, we have $b^{(n)}\notin Im({\mathcal C}_{(n)})$, so the system (\ref{equivlinear}) has no solution.
This is the essential difficulty in the analysis of our main results (comparing \cite{Imbert2015}).

\subsubsection{The desired construction}

In this part we design a modification of (\ref{equivlinear}). The basic idea is to define a perturbation of the matrix ${\mathcal C}_{(n)}$ so that the system defined by this perturbation has a uniquely solution
$Z^n$ and the resulting approximation $u_p$ possesses the desired convergence.
Let ${\mathcal A}$ be a $m_n\times p$ matrix. For a chosen column index $l$ ($1\leq l\leq p$) of ${\mathcal A}$, we use
$\{({\mathcal A})_{k_1,k_2}^l\}$ to denote all the entries on the $l$-th column of the matrix $A$, where $\{({\mathcal A})_{k_1,k_2}^l\}$ are arranged as the same order as $\alpha_n$ with the double indices ($k_1, k_2$),
namely, the number $({\mathcal A})^l_{k_1,k_2}$ denote the entry on the $F(k_1,k_2)$-row and the $l$-column of ${\mathcal A}$. For ease of notation, we define an operator $T^l_{k_1,k_2}$ acting on $m_n\times p$ matrices as follows
$$ T^l_{k_1,k_2}{\mathcal A}=\omega^2 \sum_{r=0}^{k_{1}} \sum_{j=0}^{k_{2}} \frac{\partial_{x}^{r} \partial_{y}^{j} \xi(\mathbf{r}_0)}{r ! j !} ({\mathcal A})_{k_1-r, k_2-j}^l. $$

The following lemma define the desired decomposition of the matrix ${\mathcal C}_{(n)}$.

\begin{lemma} \label{Mdecom} Let ${\mathcal C}_{(n)}$ be the matrix defined in Definition 3.1.
There exists a decomposition
\begin{equation} \label{Rdecom}
	{\mathcal C}_{(n)} = \hat{{\mathcal C}}_{(n)} + \tilde{{\mathcal C}}_{(n)}
\end{equation}	
such that $Im(\hat{{\mathcal C}}_{(n)}) \subset \mathbb{S}_n$ and $\tilde{{\mathcal C}}_{(n)}$ is recursively calculated by
\begin{equation}
	\left\{\begin{aligned}
	& (\tilde{{\mathcal C}}_{(n)})_{k_1,k_2}^l = 0, \quad k_1+k_2 \leq 4 \\
	& 2m_{k_1} (\tilde{{\mathcal C}}_{(n)})_{k_1+2,k_2}^l+ 2m_{k_2} (\tilde{{\mathcal C}}_{(n)})_{k_1,k_2+2}^l +T^l_{k_1,k_2}\tilde{{\mathcal C}}_{(n)}  \\
	& \quad = \frac{1}{k_1!k_2!}\partial_x^{k_1}\partial_y^{k_2} \mathcal{L}\varphi_l(\mathbf{r}_0), \quad \forall 2 < k_1+k_2 \leq n-2.
	\end{aligned}\right.\label{tildeM}
\end{equation}

\end{lemma}
\begin{proof} Let $\varphi$ and $\chi_{k_1,k_2}$ in (\ref{3.new-vanishing}) be replaced by $\varphi_l$ and $({\mathcal C}_{(n)})_{k_1, k_2}^l$ respectively, we get
\begin{eqnarray}
	& 2m_{k_1} ({\mathcal C}_{(n)})_{k_1+2, k_2}^l + 2m_{k_2} ({\mathcal C}_{(n)})_{k_1, k_2+2}^l +T^l_{k_1,k_2}{\mathcal C}_{(n)} \cr
	& = \frac{1}{k_1!k_2!}\partial_x^{k_1}\partial_y^{k_2} \mathcal{L}\varphi_l(\mathbf{r}_0), \quad \forall 0 \leq k_1+k_2 \leq n-2. \label{3.new10}
\end{eqnarray}
Define the $m_n\times p$ matrix $\hat{\mathcal C}_{(n)}$ such that
$$ (\hat{\mathcal C}_{(n)})_{k_1+2, k_2}^l=({\mathcal C}_{(n)})_{k_1+2, k_2}^l-(\tilde{\mathcal C}_{(n)})_{k_1+2, k_2}^l. $$
Then, combining (\ref{3.new10}) and (\ref{tildeM}), yields
$$
	 2m_{k_1} (\hat{\mathcal C}_{(n)})_{k_1+2, k_2}^l + 2m_{k_2} (\hat{\mathcal C}_{(n)})_{k_1, k_2+2}^l +T^l_{k_1,k_2}\hat{{\mathcal C}}_{(n)} = 0.$$
This means that $Im(\hat{{\mathcal C}}_{(n)}) \subset \mathbb{S}_n$ (whose definition was given in (\ref{Rspace})).
\end{proof}

By the condition (\ref{diffH-new}) and the equation (\ref{tildeM}), we deduce that
\begin{equation} \label{diffH}
	(\tilde{{\mathcal C}}_{(n)})_{k_1,k_2}^l =
	\begin{cases} {0, \quad  k_1+k_2 \leq 4} \\{\mathcal{O}(q^2\omega^{k_1+k_2-q-2}) \neq 0, \quad  4<k_1+k_2 \leq n}, \end{cases}
\end{equation}
which means that the entries of the matrix $\tilde{{\mathcal C}}$ are sufficiently small when $q$ is large.

As we will see in Remark 3.1, we have $Im(\hat{\mathcal C}_{(n)})\ni b^{(n)}$. Then we replace the unsolvable linear system (\ref{equivlinear}) by the following
solvable linear system
\begin{equation}\label{equivlinearnew}
	\left\{\begin{aligned}
	& \text { Find } Z^{(n)} \in \mathbb{C}^{p}~~~(p=2n+1) \text { such that } \\
	& \hat{{\mathcal C}}_{(n)}  Z^{(n)} = b^{(n)}.
	\end{aligned}\right.
\end{equation}

For one solution $Z^{(n)}=(z_1~~z_2\cdots z_p)^t$ of the above algebraic system, we define the approximation
\begin{equation}\label{3.interpolation}
u_p(x,y)=(\pi_p u)|_{K_0}(x,y)=\sum\limits_{l=1}^p z_l\varphi_l(x,y),\quad (x,y)\in K_0.
\end{equation}
In order to establish the accuracy of $u_p$, we needs only to estimate $|\sum\limits_{l=1}^p z_lP_{l,n}(x,y)-T_n(x,y)|$. This task will be finished in the next subsection.

\subsection{Analysis} An important problem is how $Z^{(n)}$ depends on $\omega$. To answer this problem, we need more auxiliary results.
The basic ideas are to establish a relation between the coefficient matrix $\hat{\mathcal C}_{(n)}$ and the matrix ${\mathcal C}_{(n)}^{e}$, which has simpler structures.

Throughout this section we use $\{\theta_l\}$ to denote the direction angles given in Subsection 3.1.1 and use $k=F(k_1,k_2)$ to denote the mapping defined in Subsection 3.1.2 (for double indices $(k_{1}, k_{2})$ satisfying $0\leq k_{1}+k_{2}\leq n$).

Define a $m_n$-order diagonal matrix $\Lambda_{(n)}$ such that its entry on the $k$-row and $k$-column ($k=F(k_1,k_2)$)
equals the number $\frac{(i\omega\eta(\mathbf{r}_0))^{(k_1+k_2)}}{k_1!k_2!}$.
Let ${\mathcal C}_{(n)}^{e}$ and ${\mathcal C}_{(n)}$ denote the matrices defined in Definition \ref{matrix}. The following lemma gives properties of the two matrices.

\begin{lemma} \label{matrx-decom} The matrices ${\mathcal C}_{(n)}^{e}$ and ${\mathcal C}_{(n)}$ admit the decompositions:
\begin{equation} \label{pw-goM}
	{\mathcal C}_{(n)}^e = \Lambda_{(n)}  \hat{\mathcal C}_{(n)}^e\quad\mbox{and}\quad {\mathcal C}_{(n)} = {\mathcal C}_{(n)}^e + \sum_{j=0}^{n-1} (i\omega)^j (\sum_{s = 0}^{q-2}(i\omega)^{-s} \hat{\mathcal C}_{(n)}^{s, j}),
\end{equation}
where $\hat{\mathcal C}_{(n)}^e$ and $\hat{\mathcal C}_{(n)}^{s, j}$ are two $m_n\times p$ matrices with $\omega$-independent elements. Moreover, the elements of
$\hat{\mathcal C}_{(n)}^e$ are given by
\begin{equation}\label{mnc-ref}
	(\hat{\mathcal C}_{(n)}^e)_{k_1,k_2}^l = \cos^{k_1}\theta_l\sin^{k_2}\theta_l,\quad k_1+k_2\leq n
\end{equation}
and the matrix $\hat{\mathcal C}_{(n)}^{s, j}$ has zero elements on the beginning $m_j$ rows.
\end{lemma}

\begin{proof} The decomposition of the matrix ${\mathcal C}_{(n)}^e$ can be directly obtained by the definition of $e_l(x,y)$ and Definition \ref{matrix}.
The derivatives $\partial_{x}^{k_1} \partial_{y}^{k_2} \varphi_{l,s}(\mathbf{r}_0)$ can be expanded into the sum of different scales of $(i\omega)$. Hence, for each $0\leq s\leq q-2$, we decompose the matrix ${\mathcal C}_{(n)}^s$ with different orders of $(i\omega)$:
\begin{equation}
	{\mathcal C}_{(n)}^s = a_{l,s}(\mathbf{r}_0) \cdot {\mathcal C}_{(n)}^e + \sum_{j=0}^{n-1} (i\omega)^j \hat{\mathcal C}_{(n)}^{s, j}, \label{3.new11}
\end{equation}
where the matrices $\hat{\mathcal C}_{(n)}^{s, j}$ ($ 0\leq j\leq n-1$) are all $\omega$-independent, and has zero elements on the beginning $m_{j-1}$ rows because the scale $(i\omega)^j$ of the matrix $\hat{\mathcal C}_{(n)}^{s, j}$ in (\ref{3.new11}) comes from the derivatives of the functions $\varphi_{l,s}$ ($l=1,\cdots,p$),
whose orders are not less than $j$.
Using the regulation
$$a_{l}(\mathbf{r}_0) = \sum_{s = 0}^{q-2} (i\omega)^{-s} a_{l,s}(\mathbf{r}_0) = 1$$
and substituting (\ref{3.new11}) into (\ref{first-decomp}) yields
\begin{eqnarray*}
	{\mathcal C}_{(n)} &= \sum_{s = 0}^{q-2} (i\omega)^{-s} (a_{l,s}(\mathbf{r}_0) \cdot {\mathcal C}_{(n)}^e + \sum_{j=0}^{n-1} (i\omega)^j \hat{\mathcal C}_{(n)}^{s, j})\\
		& = {\mathcal C}_{(n)}^e + \sum_{j=0}^{n-1} (i\omega)^j (\sum_{s = 0}^{q-2}(i\omega)^{-s} \hat{\mathcal C}_{(n)}^{s, j}).
\end{eqnarray*}
\end{proof}

For convenience, define a $m_n\times m_n$ matrix $$
	\mathrm{Q}_n^j = \left(\begin{array}{cc} 	0_{m_j\times m_j} & \\	& I_{(n,j)} \end{array}\right),
$$
where $I_{(n,j)}$ denotes the $(m_n-m_j)\times (m_n-m_j)$ identical matrix. Since the matrices $\hat{\mathcal C}_{(n)}^{s, j}$ has zero elements on the leading $m_j$ rows, we have
\begin{equation} \label{mnsj}
	\hat{\mathcal C}_{(n)}^{s, j} = \mathrm{Q}_n^j  \hat{\mathcal C}_{(n)}^{s, j}, \quad  0\leq j\leq n-1.
\end{equation}

The following lemma explicitly gives the rank properties of the matrix $\hat{\mathcal C}_{(n)}^e$.
\begin{lemma} \label{rankmnc} Let $\hat{\mathcal C}_{(n)}^e$ be defined by Lemma \ref{matrx-decom} and $\mathrm{Q}_n^j$ be defined above. Then
\begin{equation} \label{rank}
	Rank(\mathrm{Q}_n^j \hat{\mathcal C}_{(n)}^e) =
	\begin{cases}
		n+1, \quad & j = n-1;\\
		2n+1, \quad & j < n-1.
	\end{cases}
\end{equation}
\end{lemma}

\begin{proof} From the definition of $\mathrm{Q}_n^j$, we can see that the matrix $(\mathrm{Q}_n^j - \mathrm{Q}_n^{j+1}) \hat{\mathcal C}_{(n)}^e$ at most has $m_{j+1}-m_j$ ($=j+2$) nonzero rows ($0\leq j \leq n-1$).
Let $\hat{\mathcal C}_{(n)}^{e,j}$ denote such a nonzero $(j+2)\times(2n+1)$ sub-matrix of $(\mathrm{Q}_n^j - \mathrm{Q}_n^{j+1}) \hat{\mathcal C}_{(n)}^e$. It follows by (\ref{mnc-ref}) that
$$
	(\hat{\mathcal C}_{(n)}^{e,j})_{k, l} = (\cos~\theta_l)^{j+2-k}(\sin~\theta_l)^{k-1}, \quad 1\leq k\leq j+2,~1\leq l\leq 2n+1.
$$
Set $z_l = \cos(\theta_l) + i \sin(\theta_l) = (\cos(\theta_l) - i \sin(\theta_l))^{-1}$. Then $|z_l| = 1$, and we have
\begin{eqnarray*}
	(\hat{\mathcal C}_{(n)}^{e,j})_{k, l} &= (\cos~\theta_l)^{j+2-k}(\sin~\theta_l)^{k-1} = 2^{-(j+2)} (z_l + z_l^{-1})^{j+2-k} (z_l - z_l^{-1})^{k-1} \\
	& = 2^{-(j+2)} \sum_{k_1 = 0}^{j+2-k}\sum_{k_2 = 0}^{k-1} \frac{(j+2-k)!}{k_1!(j+2-k-k_1)!} \frac{(k-1)!}{k_2!(k-1-k_2)!} z_l^{j+1-2(k_1+k_2)}.
\end{eqnarray*}
Define $\mathrm{z}_j^k = (z_1^{j+1-k}, \cdots, z_l^{j+1-k}, \cdots, z_{2n+1}^{j+1-k})^T$. By the above equality, we deduce that
\begin{equation*}
	Im((\hat{\mathcal C}_{(n)}^{e,j})^T) = span\{\mathrm{z}_j^{k}, k=0,2,\cdots,2(j+1)\}
\end{equation*}
 and
 \begin{equation*}
 	Rank(\hat{\mathcal C}_{(n)}^{e,j}) = dim(span\{\mathrm{z}_j^{k}, k=0,2,\cdots,2(j+1)\}) = j+2.
\end{equation*}
Indeed, taking $j=n-1$ and $j = n-2$, we have
\begin{equation*}
	\begin{aligned}
	Im((\hat{\mathcal C}_{(n)}^{e,n-1})^T) &= span\{\mathrm{z}_{n-1}^{k}, k=0,2,\cdots,2n\} \\
	Im((\hat{\mathcal C}_{(n)}^{e,n-2})^T) &= span\{\mathrm{z}_{n-2}^{k}, k=0,2,\cdots,2(n-1)\} \\
	& = span\{\mathrm{z}_{n-1}^{k}, k=1,3,\cdots,2n-1\}.
	\end{aligned}
\end{equation*}
Thus the nonzero block matrix of $(\mathrm{Q}_n^{n-2} - \mathrm{Q}_n^n) \cdot \hat{\mathcal C}_{(n)}^e$ satisfies
 \begin{equation*}
	Im(((\mathrm{Q}_n^{n-2} - \mathrm{Q}_n^n)\hat{\mathcal C}_{(n)}^e)^T) = span\{\mathrm{z}_{n-1}^k, k=0,1,\cdots,2n\}.
\end{equation*}
Then we have
 \begin{equation*}
 	Rank(\hat{\mathcal C}_{(n)}^{e,j}) = 2n+1 \quad \text{for } j < n-1.
\end{equation*}
\end{proof}

The following lemma gives a relation between the matrices ${\mathcal C}_{(n)}^e$ and  $\hat{{\mathcal C}}_{(n)}$, which were
defined by Definition \ref{matrix} and Lemma \ref{Mdecom} respectively.

\begin{lemma} \label{pw-go}
	Assume that $\xi \in \mathcal{C}^{q-1}(K_0)$ for $q\geq 2$, where $K_0$ denotes a generic element with the barycenter ${\bf r}_0$. Let ${\mathcal C}_{(n)}^e$ and $\hat{{\mathcal C}}_{(n)}$ be defined by Definition \ref{matrix} and Lemma \ref{Mdecom} respectively. Then there
exists a $m_n\times m_n$ matrix $\mathrm{T}_n$ such that
	\begin{equation}
		\hat{{\mathcal C}}_{(n)} = \mathrm{T}_n  {\mathcal C}_{(n)}^e, \label{3.new12}
	\end{equation}
where $\mathrm{T}_n$ is strictly diagonally-dominant for a sufficiently large $\omega$. Moreover,
we have $\lim\limits_{\omega\rightarrow+\infty}\|T_n-I\|_{\infty}=0$.

\end{lemma}

\begin{proof}
It follows by Lemma \ref{rankmnc} that the matrix $\mathrm{Q}_n^j \hat{\mathcal C}_{(n)}^e$ is of full-row rank when $j = n-1$ and is of full-column rank when $j = n-2$. Then, for all $j \leq n-1$, we have
\begin{equation*}
Im ((\hat{\mathcal C}_{(n)}^{s, j})^T) \subset Im((\mathrm{Q}_n^{n-2} \hat{\mathcal C}_{(n)}^e)^T)
\end{equation*}
and there exists a $m_n\times m_n$ matrix $\mathrm{T}_{n,s}^j$ such that
\begin{equation*}
	\hat{\mathcal C}_{(n)}^{s, j} =  \mathrm{T}_{n,s}^j \mathrm{Q}_n^{n-2} \hat{\mathcal C}_{(n)}^e.
\end{equation*}
By (\ref{mnsj}), we further obtain
\begin{equation} \label{decomrela}
\hat{\mathcal C}_{(n)}^{s, j} = (\mathrm{Q}_n^j \mathrm{T}_{n,s}^j \mathrm{Q}_n^{n-2})\hat{\mathcal C}_{(n)}^e, \quad j < n-1.
\end{equation}
It is easy to see that the matrix $\mathrm{Q}_n^j \mathrm{T}_{n,s}^j \mathrm{Q}_n^{n-2} $ has the following structure
\begin{equation} \label{matrixdecomp}
\mathrm{Q}_n^j \mathrm{T}_{n,s}^j \mathrm{Q}_n^{n-2}  = \left(\begin{array}{cc} 	0_{m_{j}\times m_{n-2}} & 0_{m_j\times(2n+1)}\\	0_{(m_n-m_j)\times m_{n-2}} & \hat{T}_{n,s}^j \end{array}\right),
\end{equation}
where $\{\hat{T}_{n,s}^j\}_{j=0}^{n-2}$  is $\omega$-independent but only depends on $\theta_l$ and the derivatives of $\xi$ at $\mathbf{r}_0$ due to the $\omega$-independent matrices $\hat{\mathcal C}_{(n)}^{s, j}$ and $\hat{\mathcal C}_{(n)}^e$ (noting that $2n+1=m_n-m_{n-2}$).

On the other hand, due to the property (\ref{diffH}) of $\tilde{{\mathcal C}}_{(n)}$, there exists a matrix $\tilde{\mathrm{T}}_n$ such that
\begin{equation} \label{decomrela1}
\tilde{{\mathcal C}}_{(n)} = (i\omega)^{-(n+1)}\mathrm{Q}_n^4 \tilde{\mathrm{T}}_n \mathrm{Q}_n^{n-2} {\mathcal C}_{(n)}^e,
\end{equation}
where $\tilde{\mathrm{T}}_n$ is $\omega$-independent but only depends on $\theta_l$ and the derivatives of $\xi$ at $\mathbf{r}_0$.
Using the decomposition (\ref{Rdecom}) and (\ref{pw-goM}), together with (\ref{decomrela}) and (\ref{decomrela1}), the matrix $\mathrm{T}_n$ satisfying (\ref{3.new12}) can be defined as
	\begin{equation}
		T_n = I + \sum_{j = 0}^{n-1} (i\omega)^j (\sum_{s=0}^{q-2}(i\omega)^{-s}\mathrm{Q}_n^j \mathrm{T}_{n,s}^j \mathrm{Q}_n^{n-2} )\Lambda_{(n)}^{-1} - (i\omega)^{-(n+1)}\mathrm{Q}_n^4 \tilde{\mathrm{T}}_n \mathrm{Q}_n^{n-2}.
	\end{equation}
Moreover, by (\ref{matrixdecomp}) and the definition of $\Lambda_{(n)}$ we know that the matrix $T_n$ satisfies the other conditions required in this lemma.
\end{proof}

\begin{remark}
By Lemma 2 in \cite{Imbert2015} and Lemma \ref{pw-go}, the matrices ${\mathcal C}^e_{(n)}$ and $\hat{\mathcal C}_{n}$ are of full-column rank $2n+1$ (for a sufficiently large $\omega$).
 Thus $Im(\hat{\mathcal C})$ and $\mathbb{S}_n$ have the same dimension $2n+1$, which implies that $Im(\hat{\mathcal C})=\mathbb{S}_n$ ($\ni b^{(n)}$) by Lemma \ref{Mdecom}.
 This means that the system (\ref{equivlinearnew}) has a unique solution.
\end{remark}

\begin{lemma} Assume that $h$ is sufficiently small. Let $q$ be defined in Theorem 2.1 or Theorem 2.2, and let $\varphi_l$ denote a basis function defined in Definition 2.1 (associated with an element $K_0$).
Then, for any two positive integers $r,j$, we have
\begin{equation}
|\partial_x^r\partial_y^j\varphi_l({\bf r})|\leq r!j!\omega^{r+j},\quad \forall {\bf r}\in K_0. \label{3.new_basis}
\end{equation}
\end{lemma}
\begin{proof} By the definition of $\varphi_l$, it suffices to verify that
\begin{equation}
|\partial_x^r\partial_y^j(x-x_0)^{q_x}(y-y_0)^{q_y}|\leq r!j!,\quad \forall (x,y)\in K_0, \label{3.new_large_q}
\end{equation}
where $(x_0,y_0)$ is the barycenter of $K_0$, $q_x$ and $q_y$ are two positive integers satisfying $q_x+q_y\leq q+2$.

Without loss of generality, we assume that $q_x\geq r$. Then
$$
\partial_x^r(x-x_0)^{q_x}=q_x(q_x-1)\cdots(q_x+1-r)(x-x_0)^{q_x-r}. $$
When $r=q_x$, it is clear that $\partial_x^r(x-x_0)^{q_x}=r!$. Let $q_x\geq r+1$. Thus
 \begin{eqnarray*}
|\partial_x^r(x-x_0)^{q_x}|&=&q_x(q_x-1)\cdots (r+1)r\cdots(q_x+1-r)|x-x_0|^{q_x-r}\cr
&=&q_x(q_x-1)\cdots [q_x-(q_x-r-1)]r\cdots(q_x+1-r)|x-x_0|^{q_x-r}\cr
&\leq&(q_x)^{q_x-r}r!h^{q_x-r}\leq (q_xh)^{q_x-r}r!,\quad \forall (x,y)\in K_0.
\end{eqnarray*}
By the definition of $q$, we have $q+2\leq h^{-1}$ for a sufficiently small $h$, which implies that
$q_x\leq h^{-1}$ for a sufficiently small $h$. Therefore, from the above inequality we get
$$ |\partial_x^r(x-x_0)^{q_x}|\leq r!,\quad \forall (x,y)\in K_0. $$
Similarly, we have
$$
|\partial_y^j(y-y_0)^{q_y}|\leq j!,\quad \forall (x,y)\in K_0.
$$
Then we obtain (\ref{3.new_large_q}).
\end{proof}



Now we can prove Theorem \ref{interp} easily.\\
{\bf Proof of Theorem \ref{interp}}. Remark 3.1 tells us that there exists a unique solution $Z^{(n)}$ of the linear
system (\ref{equivlinearnew}). By Lemma \ref{matrx-decom} and Lemma \ref{pw-go}, we have
$$  \hat{{\mathcal C}}_{(n)}=\mathrm{T}_n\Lambda_{(n)}\hat{{\mathcal C}}^e_{(n)}. $$
Then (the matrix $\hat{\mathcal C}^e_{(n)}$ has full-column rank)
$$
	Z^{(n)} = (\hat{{\mathcal C}}^e_{(n)})^{-1}\Lambda_{(n)}^{-1}\mathrm{T}_n^{-1}b^{(n)}.
$$
Hereafter $(\hat{{\mathcal C}}^e_{(n)})^{-1}$ denotes a generalized inverse of the $m_n\times (2n+1)$ matrix $\hat{{\mathcal C}}^e_{(n)}$.
By the definitions of $b^{(n)}$ and $\Lambda_{(n)}$, together with the property $\lim\limits_{\omega\rightarrow+\infty}\|T_n-I\|_{\infty}=0$, we can deduce that (for a sufficiently large $\omega$)
$$ \|\Lambda_{(n)}^{-1}\mathrm{T}_n^{-1}b^{(n)}\|_{\infty,K_0}\leq C(\xi,n)\max\limits_{0\leq k+r\leq n}\omega^{-(k+r)}\|\partial_x^k\partial_y^ru\|_{\infty,K_0}. $$
Notice that the matrix $\hat{{\mathcal C}}^e_{(n)}$ depends only on $n$ and is independent on $\omega$.
Moreover, by the assumption (\ref{2.newregul}) we have
$$ \max\limits_{0\leq k+r\leq n}\omega^{-(k+r)}\|\partial_x^k\partial_y^ru\|_{\infty,K_0}\leq C(\xi). $$
Then we obtain
\begin{equation}
\|Z^{(n)}\|_{\infty} \leq C(\xi, n)
.\label{3.13new}
\end{equation}
On the other hand, from the definitions of the polynomials $T_n(x,y)$ and $P_{l,n}(x,y)$, for $(x,y)\in K_0$ we have
 \begin{eqnarray*}
|T_n(x, y) - \sum_{l=1}^{2n+1} z_l P_{l,n}(x,y)|&=&|\sum_{m=0}^{n} \sum_{k_{1}+k_{2}=m} (b^{(n)})_{k_{1}, k_{2}}(x-x_0)^{k_{1}}(y-y_0)^{k_{2}}\cr
&-&\sum_{l=1}^{2n+1}z_l\sum_{m=0}^{n} \sum_{k_{1}+k_{2}=m} (\mathcal {C}_{(n)})_{k_{1}, k_{2}}^l(x-x_0)^{k_{1}} (y-y_0)^{k_{2}}|.
\end{eqnarray*}
Thus, by the definition of $Z^{(n)}$ and Lemma \ref{Mdecom}, we deduce that
\begin{eqnarray*}
|T_n(x, y) - \sum_{l=1}^{2n+1} z_l P_{l,n}(x,y)|
&=&|\sum_{m=0}^{n} \sum_{k_{1}+k_{2}=m} (\hat{{\mathcal C}}_{(n)} Z^{(n)})_{k_{1}, k_{2}} (x-x_0)^{k_{1}} (y-y_0)^{k_{2}}\cr
&-&\sum_{m=0}^{n} \sum_{k_{1}+k_{2}=m} ({\mathcal C}_{(n)} Z^{(n)})_{k_1,k_2}(x-x_0)^{k_{1}} (y-y_0)^{k_{2}}| \\
&=&|\sum_{m=0}^{n} \sum_{k_{1}+k_{2}=m} (\tilde{{\mathcal C}}_{(n)} Z^{(n)})_{k_1,k_2}(x-x_0)^{k_{1}} (y-y_0)^{k_{2}}|.
\end{eqnarray*}
Then, when $n\leq 4$, (\ref{diffH}) implies that $T_n(x, y)-\sum_{l=1}^{2n+1} z_l P_{l,n}(x,y)=0$. If $n\geq 5$, using (\ref{diffH}) and (\ref{3.13new}) yields
\begin{eqnarray}\label{3.newestimate1}
|T_n(x, y) - \sum_{l=1}^{2n+1} z_l P_{l,n}(x,y)|
&\leq&C(\xi, n)q^2\sum_{k_{1}+k_{2}=5}^n {1\over k_1!k_2!}\omega^{k_1+k_2-q-2}h^{k_1+k_2}\cr
&=&C(\xi, n)q^2\sum_{m=5}^n\omega^{m-q-2}h^{m}.
\end{eqnarray}
We want to estimate the above sum.
Notice that $\omega h\leq C_0$, we need only to consider the term with $m=5$ in the sum of (\ref{3.newestimate1}).
It is easy to check that, when $q\geq {(n-4)\ln(\omega h)^{-1}\over \ln\omega}$ ($n\geq 5$ and $\omega>1$),
we have $q^2\omega^{5-q-2}h^5\leq h^{n+1}\omega^{n+1}$, and so
$$ q^2\sum_{m=5}^n\omega^{m-q-2}h^{m}\leq Ch^{n+1}\omega^{n+1}. $$
Therefore, it follows by (\ref{3.newestimate1}) that
$$ |T_n(x, y) - \sum_{l=1}^{2n+1} z_l P_{l,n}(x,y)|\leq C(\xi, n)h^{n+1}\omega^{n+1},\quad (x,y)\in K_0. $$
Using this inequality and the error estimate (\ref{utaylor}) of the Taylor polynomials $T_n(x,y)$ and $P_{l,n}$, for $(x,y)\in K_0$ we obtain
\begin{eqnarray*}
|u(x, y) - u_p(x, y)|&=& |(u(x, y) - T_n(x, y)) + (T_n(x, y) - \sum_{l=1}^{2n+1} z_l P_{l,n}(x,y)) \cr
&+& (\sum_{l=1}^{2n+1} z_l P_{l,n}(x,y) - u_p(x, y))| \cr
&\leq& C(\xi, n) h^{n+1}\omega^{n+1}.
\end{eqnarray*}
Here we have used the assumption (\ref{2.newregul}) and the inequalities
\begin{eqnarray*}
&&|\sum_{l=1}^{2n+1} z_l P_{l,n}(x,y) - u_p(x, y)|=|\sum_{l=1}^{2n+1} z_l(P_{l,n}(x,y)-\varphi_l(x,y))|\cr
&&=|\sum_{l=1}^{2n+1} z_l\sum_{r+j=n+1}{\partial_x^r\partial_y^j\varphi_l(\xi_x,\xi_y)\over r_!j_!}(x-x_0)^r(y-y_0)^j|\cr
&&\leq C(\xi,n)h^{n+1}\omega^{n+1},\quad\forall (x,y)\in K_0,
\end{eqnarray*}
where the last inequality is derived by the estimate (\ref{3.new_basis}). Then we get the first estimate in (\ref{intererr}).

The more general estimate (\ref {2.new-intererr1}) can be verified in almost the same manner, taking the Taylor formula for the gradient of $u-u_p$, up to the order $k$.\hfill$\Box$
\begin{remark} Theorem \ref{interp2} can be proved in almost the same way as Theorem \ref{interp}, with the following obvious modifications:
use the ($2n$)-order Taylor polynomials of $u$ and the plane wave basis functions $\psi_{l,j}$
and change the supper index $p$ in the sum of (\ref{3.1new}) into $2p$ since there are two basis functions for each plane wave directions. With this changes, the unknown $Z^{(n)}$ in
(\ref{equivlinear}) and  (\ref{equivlinearnew}) should be a ($2p$)-dimensional column vector; moreover, ${\mathcal C}_{(n)}$, $\tilde{\mathcal C}_{(n)}$ and $\hat{\mathcal C}_{(n)}$ become $m_{2n}\times 2p$ matrices.
\end{remark}

\begin{remark} The established order of $h$-convergence of the GOPW approximation $u_p$ is slightly higher than that of the standard plane wave approximation (which is available only for the case of piecewise constant wave numbers) since different techniques are used in this section. However, we fail to derive the desired order of $p$-convergence of $u_p$ by the techniques.
\end{remark}

\section{Variational formulation for nonhomogeneous Helmholtz equations}\label{variation}

This section devotes to the discretization of the nonhomogeneous Helmholtz equation (\ref{nonhomogeneousHelm}) by using the plane wave basis functions constructed in Section 2.
To this end, we adopt the local-global variational method first presented in \cite{Huqy2018} (this method has been extended in \cite{Yuan2019} to the case with variable wave number based on GPW basis
functions).

\subsection{A local-global variational formulation of (\ref{nonhomogeneousHelm})}

Assume that $f$ is defined in a slightly large domain containing $\Omega$ as its subdomain and decompose the solution $u$ of (\ref{nonhomogeneousHelm}) into $u = u^{(1)} + u^{(2)}$, where $u^{(1)}$ is a piecewise particular solution of (\ref{nonhomogeneousHelm}) with the homogeneous boundary condition, and $u^{(2)}$ locally satisfies homogeneous Helmholtz equations.

The following notations are inheritted from \cite{Huqy2018}. Assume that the domain $\Omega$ is strictly star-shaped. Let $\Omega$ be decomposed into the union of some subdomains in the sense that
\begin{equation*}
	\bar{\Omega} = \cup_{k = 1}^{N}\bar{\Omega}_k, \quad \Omega_l\cap\Omega_j = \emptyset \quad \forall l\neq j,
\end{equation*}
where each $\Omega_k$ is star-shaped with respect to a ball, but it may be not a polygon. Let $\mathcal{T}_h$ denote the triangulation comprising the elements $\{\Omega_k\}$, where $h$ denotes
the mesh size of this triangulation, i.e., the diameter of the biggest element. As usual, we assume that $\mathcal{T}_h$ is quasi-uniform. Define
\begin{equation*}
	\Gamma_{lj} = \partial\Omega_l\cap\partial\Omega_j \quad \forall l\neq j\quad\mbox{and}\quad\Gamma_k = \bar{\Omega}_k\cap\partial\Omega \quad (k = 1,\cdots,N).
\end{equation*}
Let $\mathcal{F}_{h}=\bigcup_{k} \partial \Omega_{k}$ denote the skeleton of the mesh. Set
$$ \mathcal{F}_{h}^{\mathrm{B}}=\mathcal{F}_{h} \cap \partial \Omega=\cup_{k=1}^N \Gamma_k $$
and
$$ \mathcal{F}_{h}^{\mathrm{I}}=\mathcal{F}_{h} \backslash \mathcal{F}_{h}^{\mathrm{B}}=\bigcup\limits_{lj}\Gamma_{lj}.$$
On every interface $\partial \Omega_{l} \cap \partial \Omega_{j}$, define
\begin{equation*}
	\begin{array}{l}
	{\text { the averages: }\{\{u\}\} :=\frac{u_{I}+u_{j}}{2}, \quad \{\{\boldsymbol{\sigma}\}\} :=\frac{\boldsymbol{\sigma}_{l}+\boldsymbol{\sigma}_{j}}{2}} \\
	{\text { the jumps: }[[u]]_{N} = \mathbf{n}_{l} \cdot u_{l} + \mathbf{n}_{j} \cdot u_{j}, \quad [[\boldsymbol{\sigma}]]_{N} = \mathbf{n}_{l} \cdot \boldsymbol{\sigma}_{l} + \mathbf{n}_{j} \cdot \boldsymbol{\sigma}_{j}}.
	\end{array}
\end{equation*}

For each element $\Omega_{k}$, let $\Omega_{k}^{*}$ be a fictitious domain that has almost the same size of $\Omega_{k}$ and contains $\Omega_{k}$ as its subdomain. Throughout this paper we assume that each $\Omega_k^{*}$ is a disc (see \cite{Huqy2018}, Remark 2.1), whose radius is chosen as $r_k\approx{h\over 2}$. Let $u^{(1)} \in L^{2}(\Omega)$ be defined as $u^{(1)}|_{\Omega_{k}}=u_{k}^{(1)}|_{\Omega_{k}}$ for each $\Omega_{k}$, where $u_{k}^{(1)} \in H^{1}(\Omega_{k}^{*})$ satisfies the nonhomogeneous local Helmholtz equation on the fictitious domain $\Omega_{k}^{*}$:
\begin{equation}\label{fem}
	\left\{
	\begin{array}{ll}
		{-\Delta u_{k}^{(1)}-\kappa^2 u_{k}^{(1)}=f} & {\text { in } \quad \Omega_{k}^{*}} \\
		{(\partial_{\mathbf{n}_{k}}+i r_k^{-1}) u_{k}^{(1)}=0} & {\text { on } \quad \partial \Omega_{k}^{*}}
	\end{array}(k=1,2, \ldots, N).
	\right.
\end{equation}
The variational formulation of (\ref{fem}) is to find $u_{k}^{(1)}\in H^1(\Omega_k^*)$ such that
\begin{equation}\label{femvf}
	\left\{
	\begin{array}{c}
		{\int_{\Omega_{k}^{*}}(\nabla u_{k}^{(1)} \cdot \nabla \overline{v}_{k}-\kappa^{2} u_{k}^{(1)} \overline{v}_{k}) d \mathbf{r}+i r_k^{-1}\int_{\partial \Omega_{k}^{*}} u_{k}^{(1)} \overline{v}_{k} d \mathbf{r}=\int_{\Omega_{k}^{*}} f \overline{v}_{k} d \mathbf{r}} \\
		{\forall v_{k} \in H^{1}\left(\Omega_{k}^{*}\right)~~(k=1,2, \ldots, N)}
	\end{array}
	\right.
\end{equation}
It is easy to see that $u^{(2)}=u-u^{(1)}$ is uniquely determined by the following homogeneous Helmholtz equations of $u_{k}^{(2)}=u^{(2)} | \Omega_{k}$ :
\begin{equation*}
	-\Delta u_{k}^{(2)}-\kappa^{2} u_{k}^{(2)}=0 \quad \text { in }~~\Omega_{k}~~~(k=1,2, \ldots, N)
\end{equation*}
with the following boundary condition on $\Gamma$ and the interface conditions on $\Gamma_{kj}$:
\begin{equation}
\begin{cases}
	\partial_{\mathbf{n}} u_{k}^{(2)}+i \eta u_{k}^{(2)}=g-(\partial_{\mathbf{n}} u_{k}^{(1)}+i\eta u_{k}^{(1)}) & \text { over $\Gamma_{k}$} \\
	u_{k}^{(2)}-u_{j}^{(2)}=-(u_{k}^{(1)}-u_{j}^{(1)}) & \text { over $\Gamma_{k j}$} \\
	\partial_{\mathbf{n}_{k}} u_{k}^{(2)}+\partial_{\mathbf{n}_{j}} u_{j}^{(2)}=-(\partial_{\mathbf{n}_{k}} u_{k}^{(1)}+\partial_{\mathbf{n}_{j}} u_{j}^{(1)}) & \text { over $\Gamma_{k j}$} \\
	(k \neq j ; k, j=1,2, \ldots, N) &
\end{cases}
\end{equation}

Define
\begin{equation*}
	V(\mathcal{T}_{h})=\{v_{h} \in L^{2}(\Omega) ;~ v_{h} \in H^1(\Omega_{k}) \text { on each } \Omega_{k}\}.
\end{equation*}
Then the standard DG method can be described as: find $u^{(2)}\in V(\mathcal{T}_{h})$ such that
\begin{equation}\label{4.defin}
	\mathcal{A}_{h}(u^{(2)}, v)=\ell_{h}(v), \quad \forall v \in V(\mathcal{T}_{h}),
\end{equation}
where
\begin{eqnarray*}
	\mathcal{A}_{h}(u, v)&=&\int_{\Omega}(\nabla_{h} u \cdot \overline{\nabla_{h} v}-\kappa^{2} u \overline{v}) d \mathbf{r}-\int_{\mathcal{F}_{h}^{1}}[[u]]_{N} \overline{\{\{\nabla_{h} v\}\}} \mathrm{d} s \\
	&+& i\omega^{-1} \int_{\mathcal{F}_{h}^{I}} \beta[[\nabla_{h} u]]_{N} \cdot \overline{[[\nabla_{h} v]]_{N}} \mathrm{d} s+ i \omega \int_{\mathcal{F}_{h}^{I}}\alpha[[u]]_{N} \cdot \overline{[[v]]}_{N} \mathrm{d}s\\
	&-&\int_{\mathcal{F}_{h}^{I}}\{\{\nabla_{h} u\}\} \cdot \overline{[[v]]_{N}} \mathrm{d} s-\omega^{-1}\int_{\mathcal{F}_{h}^{\mathrm{B}}} \delta \eta u(\overline{\nabla_{h} v \cdot \mathbf{n}}) \mathrm{d}s-\int_{\mathcal{F}_{h}^{\mathrm{B}}} \delta(\nabla_{h} u \cdot \mathbf{n}) \overline{v} \mathrm{d} s \\
	&+&i\omega^{-1} \int_{\mathcal{F}_{h}^{\mathrm{B}}} \delta(\nabla_{h} u \cdot \mathbf{n}) \overline{(\nabla_{h} v \cdot \mathbf{n})} \mathrm{d} s + i \int_{\mathcal{F}_{h}^{\mathrm{B}}}(1-\delta) \eta u \overline{v} \mathrm{d} s
\end{eqnarray*}
and
$$
	\ell_{h}(v)=\sum_{k} \int_{\Omega_{k}} f \overline{v} \mathrm{d} \mathbf{r}-\mathcal{A}_{h}(u_{h}^{(1)}, v)+i\omega^{-1} \int_{\mathcal{F}_{h}^{\mathrm{B}}} \delta g \overline{(\nabla_{h} v \cdot \mathbf{n})} \mathrm{d} s+\int_{\mathcal{F}_{h}^{\mathrm{B}}}(1-\delta) g \overline{v} \mathrm{d} s.
$$
As explained in  \cite{Imbert2017}, when the generalized plane waves are employed as the discrete space, the convergence of the standard discontinuous Galerkin method is difficult to establish. Instead a stabilizing term needs to be added to the sesquilinear form, so we define
\begin{equation}
	\mathcal{B}_{h}(u, v)=\mathcal{A}_{h}(u, v)+\frac{i\gamma}{\omega^2}\sum\limits_{k=1}^N \int_{\Omega_k}(\Delta u+\kappa^2 u) \overline{(\Delta v+\kappa^{2} v)} d \mathbf{r},
\end{equation}
where $\gamma>0$ denotes a penalty parameter.

In this paper, we make the simple choices of parameters
\begin{equation*}
	\alpha = \beta =\delta=\gamma = \frac{1}{2}.
\end{equation*}
\subsection{Discretization of the variational problem}

We first give the spectral element discretization of the local nonhomogeneous problems.
Since $\Omega_{k}^{*}$ is a sufficiently smooth domain and $f$ is smooth on $\Omega_{k}^{*}$, the solution $u_{k}^{(1)}$ possesses high regularity on $\Omega_{k}^{*}$. Moreover, the fictitious domain $\Omega_{k}^{*}$ has almost the same size as the element $\Omega_{k}$ Thus the subproblems (\ref{femvf}) should be discretized by the spectral element method, so that the resulting approximate solutions have higher accuracy.

Let $m$ be a positive integer and $D$ be a bounded and connected domain in $\mathbb{R}^{2}$. Let $S_m(D)$ denote the set of polynomials defined on $D$, whose orders are less or equal to $m$.

The discrete variational problems (see Subsection 3.1 in \cite{Huqy2018}) of (\ref{femvf}) are to find $u_{k, h}^{(1)} \in S_{m}(\Omega_{k}^{*})$ such that
\begin{equation}\label{femdvf}
\left\{
	\begin{aligned}
	& \int_{\Omega_{k}^{*}}(\nabla u_{k, h}^{(1)} \cdot \overline{\nabla v_{k, h}}-\kappa^{2} u_{k, h}^{(1)} \overline{v_{k, h}}) \mathrm{d} \mathbf{x}+i r_k^{-1}\int_{\partial \Omega_{k}^{*}}u_{k, h}^{(1)} \overline{v_{k, h}} \mathrm{d} S=\int_{\Omega_{k}^{*}} f \overline{v_{k, h}} \mathrm{d} \mathbf{x} \\
	&\forall v_{k, h} \in S_m\left(\Omega_{k}^{*}\right)(k=1,2, \ldots, N)
\end{aligned}\right.
\end{equation}
Since we choose the fictitious domain $\Omega_{k}^{*}$ to be a disc, the variational problems (\ref{femdvf}) can be solved easily by using the polar coordinate transformation for the calculation of the involved integrations. Define $u_{h}^{(1)} \in \prod_{k=1}^{N} S_{m}(\Omega_{k})$ by $u_{h}^{(1)}|_{\Omega_{k}}=u_{k, h}^{(1)} |_{\Omega_{k}}$.

We use $V_p(\mathcal{T}_{h})$
to denote the space spanned by the local GOPW basis functions $\varphi_{l}$ or $\psi_{l,j}$~($l=1,2, \ldots, p; j=1,2)$. Namely, the space $V_{p,q}(\mathcal{T}_{h})$ can be defined as
\begin{equation}\label{basisf}
V_{p,q}^{(1)}(\mathcal{T}_{h})=\{v\in L^2(\Omega):~v|_{\Omega_k}\in V_{p,q}^{(1)}(\Omega_k),~~k=1,\cdots,N\}
\end{equation}
or
\begin{equation}\label{basisf1}
V_{p,q}^{(2)}(\mathcal{T}_{h})=\{v\in L^2(\Omega):~v|_{\Omega_k}\in V_{p,q}^{(2)}(\Omega_k),~~k=1,\cdots,N\}.
\end{equation}

Now the GOPW discontinuous Galerkin method (GOPWDG) can be described as: seek $u_{h}^{(2)} \in V_{p,q}(\mathcal{T}_{h})$ such that,
\begin{equation}\label{GOPWDG}
	\mathcal{B}_{h}(u_{h}^{(2)}, v_{h})=\ell_{h}(v_{h}), \quad \forall v_{h} \in V_{p,q}(\mathcal{T}_{h})
\end{equation}
The final approximate solution $u_h\in L^2(\Omega)$ is defined by $u_h|_{\Omega_k} = u_h^{(1)}|_{\Omega_k} + u_h^{(2)}|_{\Omega_k}$.

\subsection{Error estimates of the approximate solutions}
In the rest of the paper, for a positive integer $j$ and a bounded and connected domain $D$, let $\|v\|_{j, D}$ and $|v|_{j, D}$ denote the norm and the semi-norm of $v$ on the Sobolev space $H^j(D)$, respectively.

Following the notations in \cite{Huqy2018}, let $\Omega_{\delta}$ be the union of $\Omega$ and the boundary layer with the thickness $\delta$, i.e.
\begin{equation*}
	\Omega_{\delta}=\Omega \cup\{\mathbf{x}: \operatorname{dist}(\mathbf{x}, \partial \Omega)<\delta\},
\end{equation*}
where $\delta$ is a small positive number satisfying
\begin{equation*}
	\Omega_{\delta} \supseteq \cup_{k=1}^N \Omega_k^{*}.
\end{equation*}
Hereafter $C_0$ denotes one constant, which is independent of $\omega$, $h$ and $p$.

As in \cite{Huqy2018}, we can prove the following estimate of $u^{(1)}-u_h^{(1)}$ (notice that we impose a slightly different boundary boundary in (\ref{fem}), so the original assumption
$\omega h\geq c_0$ in \cite{Huqy2018} is unnecessary).
\begin{lemma}\label{approlst}
	Let $m \geq 2$ and $2 \leq s \leq m+1$. Assume that $\omega h \leq C_0$ and $f \in H^{s-2}(\Omega_{\delta})$. Then the following error estimates hold
\begin{equation}\label{4.new-ineq0}
	(\sum_{k=1}^{N}\|u^{(1)}-u_{h}^{(1)}\|_{j, \Omega_{k}}^{2})^{\frac{1}{2}} \leq C(\xi)(\frac{h}{m})^{s-j} \sum_{l=0}^{s-2} \omega^{s-l-1}\|f\|_{l, \Omega_{\delta}} \quad (j=0,1,2).
\end{equation}
\end{lemma}

For the approximations $u_h$ defined in Subsection 4.2, we have error estimates without wave number pollution.
\begin{theorem}
	Let $m$ and $p$ be chosen such that $m \geq 2$ and $p=2n+1$ with an integer $n\geq 2$. For the first type GOPW space $V_{p,q}^{(1)}({\mathcal T}_h)$, we assume that
$2< s\leq\min\{m+1, n+1\}$ and define $q=\max\{2, [{(s-5)\ln (\omega h)^{-1}\over \ln\omega}], [{(s+{1\over 2})\ln(\omega h)^{-1})\over \ln h^{-1}}]\}$;
for the second type GOPW space $V_{p,q}^{(2)}({\mathcal T}_h)$, we assume that
$2< s\leq\min\{m+1, 2n+1\}$ and define $q=\max\{1, [{(s-5)\ln (\omega h)^{-1}\over \ln\omega}], [{(s+{1\over 2})\ln(\omega h)^{-1})\over \ln h^{-1}}]\}$.
Suppose that $f \in H^{s-2}(\Omega_{\delta})$ and $u \in H^s(\Omega)$.
If $\xi(\mathbf{r})={1\over c^2(\mathbf{r})}$ is a piecewise smooth function, then the approximation $u_h$ admits the estimate
\begin{equation}
\| u - u_h \|_{0, \Omega} \leq C(f, \xi)(\frac{\omega h}{m})^{s-2} +C(f, g, \xi, n) (\omega h)^{s-1}. \label{5.error}
\end{equation}
\end{theorem}
\begin{proof} Using the definitions of $u$ and $u_h$, we have
\begin{equation}
\| u - u_h \|_{0, \Omega} \leq\|u^{(1)}-u^{(1)}_h\|_{0,\Omega}+\|u^{(2)}-u^{(2)}_h\|_{0,\Omega}. \label{4.new-ineq1}
\end{equation}
By Lemma 4.1, we need only to estimate the second term.
Define the broken Sobolev space
$$ H^2({\mathcal T}_h)=\{ v\in
L^2(\Omega):~ v|_{\Omega_k} \in H^2(\Omega_k),~ ~~\forall
\Omega_k\in {\mathcal T}_h \}
$$
and its norm
\begin{eqnarray*}
&|||v|||_{\mathcal{F}_h}^2~:=
\omega^{-1}||\beta^{1/2}\llbracket \nabla v\rrbracket_N||_{0,\mathcal{F}_h^I}^2 +
\omega||\alpha^{1/2}\llbracket v\rrbracket_N||_{0,\mathcal{F}_h^I}^2 \cr &
+\omega^{-1}||\delta^{1/2} \nabla v\cdot {\bf
n}||_{0,\mathcal{F}_h^B}^2 +
\omega||(1-\delta)^{1/2}v||_{0,\mathcal{F}_h^B}^2 + \omega^{-2}\sum\limits_{k=1}^N||\gamma^{1/2}(\Delta + \kappa^2 I)v)||_{0,\Omega_k}^2.
\end{eqnarray*}
Following the proof of Lemma 3.7 in \cite{Hiptmair2011}, we can show the Poincare-type inequality
\begin{equation}
\|u^{(2)}-u^{(2)}_h\|_{0,\Omega}\leq C(\omega^{-{1\over 2}}h^{-{1\over 2}}+\omega^{{1\over 2}}h^{{1\over 2}})|||u^{(2)}-u^{(2)}_h|||_{\mathcal{F}_h}.\label{4.new-ineq2}
\end{equation}

Let $\pi_p: H^2({\mathcal T}_h)\rightarrow V_{p,q}({\mathcal T}_h)$ denote the interpolation operator defined by (\ref{3.interpolation}). Set $\varepsilon^{(1)}_h=u^{(1)}-u^{(1)}_h$ and $\varepsilon^{(2)}_h=u^{(2)}-\pi_pu^{(2)}$.
As in \cite{Huqy2018} and \cite{Yuan2019}, by using (\ref{4.defin}) and (\ref{GOPWDG}) we can verify that
\begin{equation}
|||u^{(2)}-u^{(2)}_h|||_{\mathcal{F}_h}\leq {1+\sqrt{2}\over 2}\big(|||\varepsilon^{(1)}_h|||_{\mathcal{F}^+_h}+|||\varepsilon^{(2)}_h|||_{\mathcal{F}^+_h}\big),
\label{4.new-estim1}
\end{equation}
where $|||\cdot|||_{\mathcal{F}^+_h}$ is an augmented norm of $|||\cdot|||_{\mathcal{F}_h}$ (adding some norms of averages on the local interfaces, see \cite{Hiptmair2011} and \cite{Yuan2019}).
By the trace inequalities and Lemma \ref{approlst}, we can deduce that
\begin{equation}
|||\varepsilon^{(1)}_h|||_{\mathcal{F}^+_h}\leq C\omega^{-{1\over 2}}h^{{1\over 2}}(\sum\limits_{k=1}^N\|u^{(1)}-u^{(1)}_h\|^2_{2,\Omega_k})^{{1\over 2}}\leq C(f,g,\xi)(\omega h)^{{1\over 2}}\big(\frac{\omega h}{m}\big)^{s-2}. \label{4.new-estim2}
\end{equation}
On the other hand, since the $L^{\infty}$ norm on an edge can be controlled by the $L^{\infty}$ norm on the two elements containing the edge, we can prove that
\begin{eqnarray}
|||\varepsilon^{(2)}_h|||_{\mathcal{F}^+_h}&\leq& C\bigg(\sum\limits_{k=1}^N\big(\omega h\|\varepsilon^{(2)}_h\|^2_{C(\Omega_k)}+\omega^{-1}h\|\varepsilon^{(2)}_h\|^2_{C^1(\Omega_k)}\big)\bigg)^{{1\over 2}}\cr
&+&C\omega^{-1}\big(\sum\limits_{k=1}^N||(\Delta + \kappa^2 I)\varepsilon^{(2)}_h||^2_{0,\Omega_k}\big)^{{1\over 2}}\label{4.new-estim3}
\end{eqnarray}
By (\ref{2.defin}) (since $q\leq (\omega h)^{-{1\over 2}}$ for small $\omega h$) and (\ref{3.13new}), we have
$$ |(\Delta+\kappa^2I)\varepsilon^{(2)}_h|=|(\Delta+\kappa^2I)\pi_pu^{(2)}|\leq C(\xi,n)h^q \quad\mbox{on}~~\Omega_k. $$
Thus
\begin{equation}
\omega^{-1}\big(\sum\limits_{k=1}^N||(\Delta + \kappa^2 I)\varepsilon^{(2)}_h||^2_{0,\Omega_k}\big)^{{1\over 2}}\leq C(\xi,n)h^q\omega^{-1}.
\label{4.new-ineq3}
\end{equation}
From the expression of $u$ and Lemma 4.1 of \cite{Huqy2018}, we have
$$ \|u^{(2)}\|^2_{C^s(\Omega_k)}\leq 2(\|u\|^2_{C^s(\Omega_k)}+\|u^{(1)}\|^2_{C^s(\Omega^*_k)})\leq C(f,g,\xi)\omega^{2s}. $$
Then $u^{(2)}$ satisfies the assumptions in Theorem \ref{interp} and Theorem\ref{interp2}. In the assumption $q\geq [{(s-5)\ln (\omega h)^{-1}\over \ln\omega}]$, we can regard $s$ as $n+1$ in Theorem \ref{interp} or as $2n+1$ in Theorem\ref{interp2}. Using Theorem \ref{interp} and Theorem\ref{interp2} (replacing $u$ by $u^{(2)}$), we get
$$ \bigg(\sum\limits_{k=1}^N\big(\|\varepsilon^{(2)}_h\|^2_{C(\Omega_k)}+h^2\|\varepsilon^{(2)}_h\|^2_{C^1(\Omega_k)}\big)\bigg)^{{1\over 2}}\leq C(f,g,\xi, n)h^{s}\omega^s. $$
Plugging this and (\ref{4.new-ineq3}) in (\ref{4.new-estim3}), leads to
\begin{equation}
|||\varepsilon^{(2)}_h|||_{\mathcal{F}^+_h}\leq  C(f, g, \xi, n)(\omega h)^{{1\over 2}} (\omega h)^{s-1}.
\label{4.new-estim4}
\end{equation}
Here we have used the assumption that $q\geq [{(s+{1\over 2})\ln(\omega h)^{-1}\over \ln h^{-1}}]$ (if $h\ll\omega^{-1}$), which implies that $h^q\omega^{-1}\leq h^{s-{1\over 2}}\omega^{s-{1\over 2}}$.
Substituting (\ref{4.new-estim2}) and (\ref{4.new-estim4}) into (\ref{4.new-estim1}) and using (\ref{4.new-ineq2}) yields
$$ \|u^{(2)}-u^{(2)}_h\|_{0,\Omega}\leq C(f,g,\xi)(\frac{\omega h}{m})^{s-2}+C(f, g, \xi, n) (\omega h)^{s-1}. $$
Finally, combing (\ref{4.new-ineq1}) with (\ref{approlst}) and the above inequality, we obtain the desired result.
\end{proof}
\begin{remark} It can be seen from (\ref{5.error}) that the proposed method is {\it weakly pollution-free}. In particular,
for fixed $p$, $q$ and $m$ satisfying the assumptions, and on quasi-uniform, shape-regular sequences of meshes, for
any $\varepsilon > 0$ there is a $\delta>0$ independent of $\omega$ such that
$$ h\omega < \delta\Longrightarrow \|u-u_h\|_0\leq\varepsilon. $$
This property is the same as that of the plane wave-type methods for Helmholtz equations with constant wave numbers (comparing \cite{Hiptmair2011} for homogeneous case and \cite{Huqy2018} for nonhomogeneous case).
\end{remark}

\section{Numerical experiences} \label{numerical}
In this section we apply the geometric optics plane wave DG method combined with local spectral elements (GOPWDG-LSFE) introduced in Section 4
to solving the nonhomogeneous Helmholtz equations with variable wave numbers
\begin{equation}\label{numnonhomo}
\left\{\begin{aligned}
	&-\Delta u - \omega^2\xi u = f \quad \text{in $\Omega$}, \\
	&(\frac{\partial}{\partial \mathbf{n}} + i\omega\sqrt{\xi}) u = g \quad \text{on $\partial\Omega$},
\end{aligned}\right.
\end{equation}
 and we report some numerical results to illustrate the effectiveness of the proposed methods.

Let $\Omega$ be divided into small rectangles. Each rectangle has the same mesh size $h$, where $h$ is the length of the longest edge of the elements. The resulting uniform triangulation is denoted by $\mathcal{T}_h$.
We choose the same number $p$ of discretized plane wave directions in every element. Let $n_e$ denote the number of basis functions per element.

Throughout this section, we use $p$ and $m$ to denote the number of plane wave directions and the order of local spectral elements respectively. As in \cite{Huqy2018}, we choose the optimal $p$ and $m$ satisfying
$p \approx 2m+1$ for the plane wave basis functions of Case 1. For the plane wave basis functions of Case 2, we choose $p$ and $m$ satisfying $p \approx m$ such that the two cases have almost the same numbers of basis functions.

We introduce the relative $L^2$ error
\begin{equation*}
	\text{Err.} = \frac{\|u_{ex} - u_h\|_{L^2(\Omega)}}{\|u_{ex}\|_{L^2(\Omega)}},
\end{equation*}
where $u_{ex}$ is the analytic solution and $u_h$ is the numerical solution. Define $\delta$ to measure the ``pollution effect" in our discrete method by
\begin{equation*}
	\delta = \frac{\log(\text{Err}_2/\text{Err}_1)}{\log(\omega_2/\omega_1)}.
\end{equation*}

\subsection{Example 1: with single wave in a heterogeneous medium}\label{example1}
We consider an example in a heterogeneous medium in the domain $\Omega = [0,1]\times[0,1]$ (see \cite{Engquist2011}): Define $\xi(\mathbf{r}) = \frac{1}{c(\mathbf{r})^2}$ with the velocity field $c(\mathbf{r})$  as a smooth converging lens with a Gaussian profile at the center $(x_0, y_0) = (1/2, 1/2)$
$$c(x, y)=\frac{4}{3}\left(1-\frac{1}{8} \exp \left(-32\left(\left(x-r_{1}\right)^{2}+\left(y-r_{2}\right)^{2}\right)\right)\right). $$

The analytic solution of the problem is given by
$$u_{ex} (x, y) = c(x, y)e^{i\omega xy}.$$
Then the source term is $f_{ex} = -\Delta u_{ex} - \omega^2\xi(\mathbf{r}) u_{ex}$ and the boundary function is chosen as $g_{ex} = (\frac{\partial}{\partial\mathbf{n}} + i\omega)u_{ex}$.

For the first type GOPW basis functions defined by (\ref{bas2}), we choose $q = 2$. The numerical results in Table \ref{Exam2GOPW2herr} show that setting $q = 2$ is enough to obtain the $(n+1)$-order $h$-convergence
for two different values of $p=2n+1$ of the plane wave basis functions.

\begin{table}[H]
\caption{Errors of approximations with respect to $h$ ($\omega = 256$): use the GOPWs of Case 1 ($q=2$).}
\label{Exam2GOPW2herr}
\begin{tabular}{ccccc}
\hline
\multicolumn{1}{l}{} & \multicolumn{2}{c}{$p=9, m=4$}  & \multicolumn{2}{c}{$p=11, m=5$}                        \\ \hline
h                    & \multicolumn{1}{c}{Err.} & \multicolumn{1}{c}{Order} & \multicolumn{1}{c}{Err.} & \multicolumn{1}{c}{Order}  \\ \hline
$\frac{1}{64}$       &           1.749e-1              &           $-$                &            3.207e-2             &           $-$         \\ \hline
$\frac{1}{128}$      &          3.892e-3               &         5.49               &            3.851e-4              &          6.38                  \\ \hline
$\frac{1}{256}$      &          9.348e-5                &        5.38               &            4.961e-6              &          6.28         \\ \hline
$\frac{1}{512}$      &          2.583e-6             &           5.18               &            7.10e-8              &          6.13       \\ \hline
\end{tabular}
\end{table}

For the second type GOPW basis functions defined by (\ref{bas1}), we set $q= 1$, and choose the number $p$ of the plane wave directions and the order number $m$
of polynomials in the local spectral space as $p = 5,7$ ($n = 2, 3$) and $m = 5,7$, respectively. Then we have the number $n_e = 2p$ of the plane wave basis functions each element. The resulting relative $L^2$ norm errors of the approximations generated by the GOPWDG-LSFE method
are listed in Table \ref{Exam2GOPW1herr}.

\begin{table}[H]
\caption{Errors of approximations with respect to $h$ ($\omega = 256$): use the GOPWs of Case 2 ($q=1$).}
\label{Exam2GOPW1herr}
\begin{tabular}{ccccc}
\hline
\multicolumn{1}{l}{} & \multicolumn{2}{c}{$p=5$, $m=5$}                      & \multicolumn{2}{c}{$p=7$, $m=7$}                        \\ \hline
h                    & \multicolumn{1}{c}{Err.} & \multicolumn{1}{c}{Order} & \multicolumn{1}{c}{Err.} & \multicolumn{1}{c}{Order} \\ \hline
$\frac{1}{64}$       &            1.898e-2              &            $-$               &            9.282e-4               &          $-$                 \\ \hline
$\frac{1}{128}$      &           4.217e-4               &         5.49                 &           5.408e-6               &           7.42               \\ \hline
$\frac{1}{256}$      &           9.981e-6               &         5.40                  &           3.415e-8               &           7.31                \\ \hline
$\frac{1}{512}$      &            2.632e-7              &         5.25                  &           2.356e-10                        &       7.18               \\ \hline
\end{tabular}
\end{table}

It shows that setting $q=1$ and $p = 2n+1$ can also obtain the approximations with $(2n+1)$-order $h$-convergence.
The data listed in the above two tables indicate that the orders of $h$-convergence of the proposed methods are slightly higher than the theoretical results.

Now we increase $\omega$ and decrease $h$ such that $\omega h=1$ or $\omega h=2$, and investigate the wave number pollution of the proposed methods.
Set $p =11, 13$ for the GOPWs of Case 1 and $p = 5, 7$ for the GOPWs of Case 2. The data are listed in Table \ref{Exam2bas2pollute} and Table \ref{Exam2bas1pollute}.

\begin{table}[H]
\caption{Little pollution effect: use the GOPWs of Case 1}
\label{Exam2bas2pollute}
\begin{tabular}{ccccc}
\hline
\multicolumn{1}{l}{}                & \multicolumn{2}{c}{$\omega h = 1$, $p = 9$, $m = 4$}    & \multicolumn{2}{c}{$\omega h = 2$, $p = 11$, $m = 5$} \\ \hline
$\omega$             & \multicolumn{1}{c}{Err.} & \multicolumn{1}{c}{$\delta$} & \multicolumn{1}{c}{Err.} & \multicolumn{1}{c}{$\delta$}\\ \hline
$128$ &            1.042e-4              &      $-$          &             5.451e-4             &     $-$      \\ \hline
$256$ &            9.338e-5              &     -0.158      &             5.452e-4             &    0.008       \\ \hline
$512$ &            9.035e-5              &     -0.048      &             5.451e-4             &     -0.003      \\ \hline
\end{tabular}
\end{table}

\begin{table}[H]
\caption{Little pollution effect: use the GOPWs of Case 2}
\label{Exam2bas1pollute}
\begin{tabular}{ccccc}
\hline
\multicolumn{1}{l}{}                & \multicolumn{2}{c}{$\omega h = 1$, $p = m = 5$}    & \multicolumn{2}{c}{$\omega h = 2$, $p = m = 7$} \\ \hline
$\omega$             & \multicolumn{1}{c}{Err.} & \multicolumn{1}{c}{$\delta$} & \multicolumn{1}{c}{Err.} & \multicolumn{1}{c}{$\delta$}\\ \hline
$128$ &          1.042e-5                &        $-$        &                  3.836e-4        &       $-$    \\ \hline
$256$ &         9.979e-6                 &      -0.067          &             3.850e-4         &        0.006   \\ \hline
$512$ &          9.805e-6                &      -0.025          &             3.843e-4             &    -0.003       \\ \hline
\end{tabular}
\end{table}

It shows that the proposed GOPWDG-LSFE methods are weakly pollution-free.

\subsection{Example 2: with two waves in a heterogeneous medium
}\label{example2}
We provide an example in a heterogeneous medium with wave speed of constant gradient (see \cite{Fomel2009}): $\xi(\mathbf{r}) = c_0^2+2\mathbf{G}_0\cdot(\mathbf{r}-\mathbf{r}_0)$ with parameters $c_0 = 1$, $\mathbf{G}_0 = (0.1, -0.2)$ and $\mathbf{r}_0 = (-0.1, -0.1)$. Referring to \cite{Fomel2009}, there are two wave rays crossing in the domain $\Omega = [0,1]\times[0,1]$. The two phase functions are known analytically and they are given by
\begin{equation}
	\phi_j = \bar{c}\sigma_j - \frac{|\mathbf{G}_0|^2}{6}\sigma_j^3, \quad j = 1,2,
\end{equation}
where
\begin{equation}
	\sigma_j = \frac{\sqrt{2(\bar{c} + (-1)^j\sqrt{\bar{c}^2-|\mathbf{G}_0|^2|\mathbf{r}-\mathbf{r}_0|^2})}}{|\mathbf{G}_0|^2}, \quad j = 1,2,
\end{equation}
with
\begin{equation}
	\bar{c} = c_0 + \mathbf{G}_0\cdot(\mathbf{r}-\mathbf{r}_0).
\end{equation}
Then the exact solution is given by
 \begin{equation}
	u_{ex} = exp(i\omega\phi_1)/(xy+i) + exp(i\omega\phi_2)/(x^2+y^2+i).
\end{equation}
The source term is $f_{ex} = -\Delta u_{ex} - \omega^2\xi(\mathbf{r}) u_{ex}$ and the boundary function is chosen as $g_{ex} = (\frac{\partial}{\partial\mathbf{n}} + i\omega)u_{ex}$.

For the first type GOPW basis functions defined by (\ref{bas2}), we choose $q = 2$. The numerical results in Table \ref{Exam3GOPW2herr} show that setting $q = 2$ is enough to obtain the $(n+1)$-order $h$-convergence
for two different values of the number $p=2n+1$ of the plane wave basis functions.

\begin{table}[H]
\caption{Errors of approximations with respect to $h$ ($\omega = 256$): use the first type of GOPW basis functions.}
\label{Exam3GOPW2herr}
\begin{tabular}{ccccc}
\hline
\multicolumn{1}{l}{} & \multicolumn{2}{c}{$p=9$, $m=4$}                      & \multicolumn{2}{c}{$p=11$, $m=5$} \\ \hline
h                    & \multicolumn{1}{c}{Err.} & \multicolumn{1}{c}{Order} & \multicolumn{1}{c}{Err.} & \multicolumn{1}{c}{Order} \\ \hline
$\frac{1}{64}$       &          1.108e-1                &          $-$                 &           3.348e-3               &           $-$           \\ \hline
$\frac{1}{128}$      &         9.424e-4                 &         6.87                  &        2.570e-5                  &         7.02              \\ \hline
$\frac{1}{256}$      &         1.713e-5                 &          5.78                 &        2.470e-7                  &         6.70                \\ \hline
$\frac{1}{512}$      &         3.492e-7                 &          5.61                 &        2.839e-9                  &         6.44                \\ \hline
\end{tabular}
\end{table}

For the second type GOPW basis functions defined by (\ref{bas1}), we choose $q = 1$ and $p=5,7$, where the number of the plane wave basis functions is $n_e = 2p$. The degree $m$
of polynomials in the local spectral space is set as $m = 5$ and $m=7$ when $p = 5$ and $p=7$ respectively. We fix $\omega=256$, but decrease the mesh size $h$. The resulting relative $L^2$ norm errors of the approximations generated by the GOPWDG-LSFE method are listed in Table \ref{Exam3GOPW1herr}.
Setting $q=1$ and $p=2n+1$ can also obtains the $(2n+1)$-order $h$-convergence for two different values of $p$.

\begin{table}[!h]
\caption{Errors of approximations with respect to $h$ ($\omega = 256$): use the second type GOPW basis functions.}
\label{Exam3GOPW1herr}
\begin{tabular}{cllll}
\hline
\multicolumn{1}{l}{} & \multicolumn{2}{c}{$p=5$, $m=5$}                      & \multicolumn{2}{c}{$p=7$, $m=7$}                        \\ \hline
h                    & \multicolumn{1}{c}{Err.} & \multicolumn{1}{c}{Order} & \multicolumn{1}{c}{Err.} & \multicolumn{1}{c}{Order} \\ \hline
$\frac{1}{64}$       &             1.898e-2             &           $-$                &          9.282e-4              &              $-$             \\ \hline
$\frac{1}{128}$      &            4.217e-4              &         5.49                  &     5.412e-6                     &          7.42                 \\ \hline
$\frac{1}{256}$      &            9.981e-6              &         5.40                  &      3.509e-8                    &          7.27                 \\ \hline
$\frac{1}{512}$      &            2.647e-7              &         5.24                  &       2.467e-10               &             7.16         \\ \hline
\end{tabular}
\end{table}

The data listed in the above two tables indicate that the orders of $h$-convergence of the proposed methods are slightly higher than the theoretical results.

Now we increase $\omega$ and decrease $h$ such that $\omega h=1$ or $\omega h=2$, and investigate the wave number pollution of the proposed methods. The numerical results are listed in Tables \ref{Exam3bas2pollute}-\ref{Exam3bas1pollute}.

\begin{table}[H]
\caption{Little pollution effect: use the first type GOPW basis functions}
\label{Exam3bas2pollute}
\begin{tabular}{ccccc}
\hline
\multicolumn{1}{l}{}                & \multicolumn{2}{c}{$\omega h = 1$, $p = 9$, $m = 4$}    & \multicolumn{2}{c}{$\omega h = 2$, $p = 11$, $m = 5$} \\ \hline
$\omega$             & \multicolumn{1}{c}{Err.} & \multicolumn{1}{c}{$\delta$} & \multicolumn{1}{c}{Err.} & \multicolumn{1}{c}{$\delta$}\\ \hline
$128$ &            1.704e-5              &       $-$         &               2.570e-5           &    $-$       \\ \hline
$256$ &            1.713e-5              &      0.007          &               2.570e-5           &      0.001     \\ \hline
$512$ &            1.602e-5              &      -0.097          &               2.568e-5           &     -0.002      \\ \hline
\end{tabular}
\end{table}

\begin{table}[H]
\caption{Little pollution effect: use the second type GOPW basis functions}
\label{Exam3bas1pollute}
\begin{tabular}{ccccc}
\hline
\multicolumn{1}{l}{}                & \multicolumn{2}{c}{$\omega h = 1$, $p = m = 5$}    & \multicolumn{2}{c}{$\omega h = 2$, $p = m = 7$} \\ \hline
$\omega$             & \multicolumn{1}{c}{Err.} & \multicolumn{1}{c}{$\delta$} & \multicolumn{1}{c}{Err.} & \multicolumn{1}{c}{$\delta$}\\ \hline
$128$ &         4.282e-6                 &       $-$         &             9.614e-7             &    $-$       \\ \hline
$256$ &         4.361e-6                 &       0.026         &             8.048e-7             &    -0.257       \\ \hline
$512$ &         4.417e-6                 &       0.018         &             8.103e-7             &    0.010       \\ \hline
\end{tabular}
\end{table}

Tables \ref{Exam3bas2pollute}-\ref{Exam3bas1pollute} show that the GOPWDG-LSFE methods are weakly pollution-free.

\subsection{GOPWDG-LSFE method versus GPWDG-LSFE method}
In this subsection, we compare the numerical performances of the proposed GOPWDG-LSFE method and the {\it generalization plane wave} DG method combined with local spectral elements (GPWDG-LSFE)
for the nonhomogeneous Helmholtz equations (\ref{numnonhomo}). We increase $\omega$ and decrease $h$ such that $\omega h=1$.
We choose the same number $p$ of discretized plane wave directions in each element and the same order $m$ of local spectral elements such that the discrete systems have the same degrees of freedom for the two methods.
We compare the $L^2$ errors of the approximations generated by the GOPWDG-LSFE method and the GPWDG-LSFE method. The numerical results are listed in Table \ref{2Exam2GOPW2herr} and
Table \ref{2Exam3GOPW2herr} for Example 1 and Example 2 respectively.

\begin{table}[H]
\caption{Errors of approximations: use the first type GOPWs ($q=2$) and GPWs ($q=5$) respectively. Example 1 defined in Subsection \ref{example1}.}
\label{2Exam2GOPW2herr}
\begin{tabular}{ccccc}
\hline
\multicolumn{1}{l}{} & \multicolumn{2}{c}{GOPW~($p=11$, $m=5$)}           & \multicolumn{2}{c}{GPW~($p=11$, $m=5$)}             \\ \hline
$\omega$                    & \multicolumn{1}{c}{Err.} & \multicolumn{1}{c}{$\delta$} & \multicolumn{1}{c}{Err.} & \multicolumn{1}{c}{$\delta$}\\ \hline
$128$      &         4.971e-6                  &        $-$                     &     2.831e-5               &       $-$ \\ \hline
$256$      &         4.958e-6                  &        0.004                 &     3.929e-5               &    0.473 \\ \hline
$512$      &         4.941e-6                 &        -0.003                 &     5.468e-5               &    0.477 \\ \hline
\end{tabular}
\end{table}

\begin{table}[H]
\caption{Errors of approximations: the first type GOPWs ($q=2$) and GPWs ($q=5$) respectively. Example 2 defined in Subsection \ref{example2}. }
\label{2Exam3GOPW2herr}
\begin{tabular}{ccccc}
\hline
\multicolumn{1}{l}{} & \multicolumn{2}{c}{GOPW~($p=11$, $m=5$)}           & \multicolumn{2}{c}{GPW~($p=11$, $m=5$)}             \\ \hline
$\omega$                    & \multicolumn{1}{c}{Err.} & \multicolumn{1}{c}{$\delta$} & \multicolumn{1}{c}{Err.} & \multicolumn{1}{c}{$\delta$}\\ \hline
$128$      &        2.439e-7                  &        $-$           		&     2.248e-6                     &       $-$ \\ \hline
$256$      &        2.428e-7                  &        -0.007                 &     3.176e-6               &    0.498 \\ \hline
$512$      &         2.442e-7                 &        0.008                 &     4.452e-6               &   0.487 \\ \hline
\end{tabular}
\end{table}
The data indicate that,
using the same DOFs and $n_e$, the GOPWDG-LSFE method has much smaller $L^2$ errors than the GPWDG-LSFE method.

\subsection{GOPWDG-LSFE method versus the high-order FEM method}
In this subsection we test the previous example 1 to compare numerical performances of GOPWDG-LSFE method (using the basis functions of Case 1) and the high-order FEM method.
Let $k\in \mathbb{N}$ denote the order of polynomials in the consider finite element space. Denote by DOFs the freedoms of the resulting discretized linear systems of the considered methods.

\begin{table}[H]
\centering
\caption{Error comparision: fixing $\omega h = 2$ and increasing $\omega$.}
\label{varyrayu4}
\begin{tabular}{cccccc}
\hline
\multicolumn{1}{l}{}         & $\omega$  & $64$ & $128$ & $256$ & $512$ \\ \hline
\multirow{2}{*}{\begin{tabular}[c]{@{}c@{}}FEM\\ $k=3$\end{tabular}}         & DOFs. &       9409         &        37249         &      148225           &        591361         \\ \cline{2-6}
                             & Err. &        1.234e-2        &      1.286e-2           &       1.356e-2          &        1.480e-2         \\ \hline
\multirow{2}{*}{\begin{tabular}[c]{@{}c@{}}GOPW\\ $m=4, p=9$\end{tabular}} & DOFs. &      9216          &         36864        &    147456     &         589824        \\ \cline{2-6}
                             & Err. &      3.656e-3          &        3.883e-3         &       3.892e-3          &      3.887e-3          \\ \hline
\end{tabular}
\end{table}

It shows that numerical solutions of the proposed GOPWDG-LSFE method posesses much less approximation errors than the high-order FEM method, with almost the same DOFs (choosing same $\omega$ and $h$).

\section{Conclusion}
In this paper we have defined new plane wave type basis functions based on the geometrical optics anasatz for the two dimensional Helmholtz equations with piecewise smooth coefficients.
We have proved best approximate properties of the resulting plane wave spaces. Furthermore we have introduced the GOPW methods combined with local spectral elements for discretization of nonhomogeneous Helmholtz equations with variable coefficients and derived weakly pollution-free error estimates of the resulting approximate solutions. We have also reported some numerical results to illustrate that the approximate solutions generated by the GOPWDG-LFSE method possess high order $h$-convergence.

\bibliographystyle{siamplain}

\end{document}